\documentclass[12pt]{amsart}
\usepackage{SSdefn}
\usepackage{comment}

\newcommand{\gen}{{\rm gen}}

\newcommand{\fgen}{{\rm fg}}
\newcommand{\dfg}{{\rm dfg}}

\newcommand{\ex}{\mathrm{ex}}

\newcommand{\FI}{\mathbf{FI}}
\newcommand{\FS}{\mathbf{FS}}

\author{Steven V Sam}
\address{Department of Mathematics, University of Wisconsin, Madison, WI}
\email{\href{mailto:svs@math.wisc.edu}{svs@math.wisc.edu}}
\urladdr{\url{http://math.wisc.edu/~svs/}}
\thanks{SS was supported by a Miller research fellowship and NSF grant DMS-1500069.}

\author{Andrew Snowden}
\address{Department of Mathematics, University of Michigan, Ann Arbor, MI}
\email{\href{mailto:asnowden@umich.edu}{asnowden@umich.edu}}
\urladdr{\url{http://www-personal.umich.edu/~asnowden/}}
\thanks{AS was supported by NSF grants DMS-1303082 and DMS-1453893 and a Sloan Fellowship.}


\title{Hilbert series for twisted commutative algebras}

\date{November 1, 2017}

\begin{document}

\maketitle

\begin{abstract}
Suppose that for each $n \ge 0$ we have a representation $M_n$ of the symmetric group $S_n$. Such sequences arise in a wide variety of contexts, and often exhibit uniformity in some way. We prove a number of general results along these lines in this paper: our prototypical theorem states that if $\{M_n\}$ can be given a suitable module structure over a twisted commutative algebra then the sequence $\{M_n\}$ follows a predictable pattern. We phrase these results precisely in the language of Hilbert series (or Poincar\'e series, or formal characters) of modules over tca's.
\end{abstract}

\tableofcontents

\section{Introduction}

Suppose that for each $n \ge 0$ we have a complex representation $M_n$ of the symmetric group $S_n$. Such sequences of symmetric group representations arise in a variety of contexts, and naturally occurring examples tend to exhibit some kind of uniformity. For example:
\begin{enumerate}
\item Fix a manifold $X$ of dimension $\ge 2$ and a non-negative integer $i$. Let $M_n$ be the $i$th cohomology group of the configuration space of $n$ labeled points on $X$. This example was studied in \cite{fimodule}. One of the theorems in  loc.\ cit.\ states that, under mild assumptions on $X$, there is a single character polynomial that gives the character of $M_n$ for all sufficiently large $n$.
\item Fix $p$ and $r$. Let $M_n$ be the space of $p$-syzygies of the Segre embedding $(\bP^r)^n \to \bP^{rn+n-1}$. This case was studied in \cite{delta}. One of the theorems in loc.\ cit.\ states that the generating function of the sequence $\{\dim(M_n)\}$ is rational.
\item Suppose that $N$ is a functor associating to each finite dimensional vector space $V$ a module $N(V)$ over the ring $\Sym(\bC^d \otimes V)$, for some fixed $d$ (and that satisfies some technical conditions). For example, one could take $N(V)$ to be the coordinate ring of the $r$th determinantal variety for some $0 \le r \le d$. Now let $M_n$ be the Schur--Weyl dual of the degree $n$ piece of $N$. This example is studied in \cite{delta} and, in much greater detail, \cite{symu1}. Again, it is known that the generating function of the dimension sequence is rational.
\item Let $E$ be a representation of a reductive group $G$, and take $M_n=(E^{\otimes n})^G$. The generating function of $\{\dim(M_n)\}$ is D-finite (and typically not rational, or even algebraic). This is presumably well-known, though we do not know a reference. See \S \ref{ss:invariants} for a proof.
\end{enumerate}
In this paper, we study the uniformity properties of sequences $\{M_n\}$ with the aim of strengthening and generalizing results like those mentioned above. We prove a number of results of the form ``if $\{M_n\}$ admits a suitable module structure over a twisted commutative algebra then the sequence $\{M_n\}$ follows a predictable pattern.'' The specific results give various precise meanings to ``suitable'' and ``predictable pattern.''

\subsection{Statement of results}

We now state some of our results in a little more detail. Let $\bk$ be a field of characteristic~0. A {\bf twisted commutative algebra} (tca) over $\bk$ is an associative unital graded $\bk$-algebra $A=\bigoplus_{n \ge 0} A_n$ equipped with an action of the symmetric group $S_n$ on $A_n$, such that the multiplication is commutative up to a ``twist'' by the symmetric group. A {\bf module} over a tca $A$ is a graded vector space $M=\bigoplus_{n \ge 0} M_n$ equipped with an action of $S_n$ on $M_n$ and a multiplication $A_n \times M_m \to M_{n+m}$ satisfying suitable axioms. Note that a module $M$ gives rise to a sequence of representations $\{M_n\}$ as discussed above. TCA's and their modules have been a central object of study in the developing field of representation stability. For example, $\bk[t]$ can be regarded as a tca (with $t$ of degree~1, and all $S_n$ actions trivial), and modules over it are equivalent to the $\FI$-modules of Church--Ellenberg--Farb \cite{fimodule}.

Suppose that $M$ is a module over a tca. We define its {\bf Hilbert series} by
\begin{displaymath}
\rH_M(t) = \sum_{n \ge 0} \dim(M_n) \frac{t^n}{n!}.
\end{displaymath}
The following theorem was proved in \cite{delta}, and later reproved in \cite{catgb}:

\begin{theorem} \label{thm:intro}
Suppose $M$ is a finitely generated module over a tca finitely generated in degree~1. Then $\rH_M(t)$ is a polynomial in $t$ and $e^t$.
\end{theorem}

This theorem implies that the generating function for $\{\dim{M_n}\}$ is rational, but is even stronger than this. The results of this paper generalize and strengthen this theorem in various ways.

We now informally describe some of the main results of this paper. In what follows, $A$ denotes a tca generated in degree~1 and $M$ denotes a finitely generated $A$-module.
\begin{enumerate}
\item We introduce the {\bf formal character} $\Theta_M$ of $M$. This records the character of each representation $M_n$, and thus contains much more information than the Hilbert series. Using the structure theory of $A$-modules developed in \cite{symu1}, we give a very precise description of $\Theta_M$. Our result shows that there is a finite expression for $\Theta_M$, and moreover that the pieces in this expression reflect the structure of $M$. As a corollary, we see that if $M$ and $N$ are finitely generated $A$-modules such that $M_n$ and $N_n$ have the same character for all $n$, then $M$ and $N$ represent the same class in the Grothendieck group of $A$-modules.
\item Using the method of \cite{delta}, we give a different proof of a rationality result for $\Theta_M$. Actually, we work with the enhanced Hilbert series introduced in \cite{symc1}, which is equivalent to the formal character. This is less precise than the proof described in (a), but does not rely on the theory from \cite{symu1}.
\item We show how the Fourier transform from \cite{symu1} affects (enhanced) Hilbert series.
\item We explain how all of the results on (enhanced) Hilbert series carry over to the more subtle Poincar\'e series.
\item Theorem~\ref{thm:intro} can be stated equivalently as: there exists a polynomial $p(T)$ with non-negative integer roots such that $p(\frac{d}{dt}) \rH_M(t)=0$. We prove a categorification of this result, in which $\frac{d}{dt}$ is replaced with the Schur derivative.
\end{enumerate}
Additionally, we prove a result for tca's not necessarily generated in degree~1:
\begin{enumerate}
\setcounter{enumi}{5}
\item If $A$ is a finitely generated tca with $A_0=\bC$ and $M$ is a finitely generated and bounded $A$-module, then $\rH_M(t)$ is a D-finite power series.
\end{enumerate}

\subsection{Open problems}

We now list some questions and open problems related to the work in this paper.
\begin{itemize}
\item To what extent do the results here carry over to positive characteristic? Theorem~\ref{thm:intro} is known to hold in positive characteristic by \cite[Corollary 7.1.7]{catgb}, but the other results of this paper are not known.
\item Result~(a) above gives a finite expression for the formal character $\Theta_M$ in the case where $M$ is a finitely generated module over a tca $A$ finitely generated in degree~1. Is there a more general result if one assumes that $A$ is bounded instead of generated in degree~1?
\item Result~(e) categorifies Theorem~\ref{thm:intro}. Can the rationality theorem for the formal character be similarly categorified? Presumably, such a categorification should make use of the more general Schur derivatives $\sD_{\lambda}$.
\item The discussion in \S \ref{ss:catdiffop} raises a number of questions about the categorified rationality. For example: in the tensor category of differential operators, what can one say about the ideal annihilating a given module? Result~(e) ensures that it is non-zero.
\item What can be said about the Hilbert series of a finitely generated module over an arbitrary finitely generated tca? For example, is it D-finite? 
\item There are interesting sequences $\{M_n\}$ of symmetric group representations that do not come from tca's, but related structures, such as $\FS^{\op}$-modules (see \cite[\S 8]{catgb}). To what extent do the results here extend to those sequences?
\end{itemize}

\subsection{Outline}

In \S \ref{s:prelim} we review some background material that we will require. In \S \ref{s:fchar} we introduce the formal character and prove our ``rationality'' theorem for it. In \S \ref{s:hilbstd} we translate our results about formal characters to Hilbert series (and Poincar\'e series), and also give an elementary proof of the rationality theorem in this setting. In \S \ref{s:cat} we prove a result categorifying the rationality theorem. In \S \ref{s:dfin} we prove D-finiteness for Hilbert series over bounded tca's not necessarily generated in degree~1. Finally, in \S \ref{s:ex} we give some examples and applications of our results.

\section{Preliminaries} \label{s:prelim}

This paper is a continuation of \cite{symu1} and hence will use the same notation.

Throughout, $X$ denotes a separated, noetherian $\bC$-scheme of finite Krull dimension, and $\cE$ is a vector bundle over $X$. We will primarily be interested in the case when $X$ is a point, but in order to simplify some of the arguments, it will be necessary to allow $X$ to be a Grassmannian. Any extra generality beyond that will not be used, but does not present any additional difficulties. The reader unfamiliar with the language of sheaf theory can assume $X$ is a point for most of the paper, in which case quasi-coherent sheaves specialize to complex vector spaces, and coherent sheaves specialize to finite-dimensional vector spaces.

Throughout, we make use of the category $\cV_X$, which has several equivalent models. The two that we use here are the category of sequences of symmetric group representations on quasi-coherent $\cO_X$-modules, and the category of polynomial functors from the category of vector spaces to quasi-coherent $\cO_X$-modules. See \cite[\S 5]{expos} for more details in the case $X = \Spec(\bC)$. In the first model, every object $\cV_X$ admits a decomposition $\bigoplus_{\lambda} \bM_{\lambda} \otimes \cF_{\lambda}$ where $\bM_{\lambda}$ is the usual Specht module over $\bC$ and $\cF_{\lambda}$ is a quasi-coherent sheaf on $X$; in the second model, we get a similar decomposition, but with $\bM_{\lambda}$ replaced by the Schur functor $\bS_{\lambda}$. Given $F \in \cV_X$ in the second model, we use $\ell(F)$ to denote the maximum number of parts of any $\lambda$ such that $\bS_\lambda$ has a nonzero multiplicity space. We say that $F$ is bounded if $\ell(F) < \infty$. We let $\cV_X^\dfg$ denote the subcategory of $\cV_X$ where the multiplicity spaces $\cF_{\lambda}$ are coherent. (The ``dfg'' superscript means ``degreewise finitely generated.'') We let $\cV_X^\fgen$ be the subcategory of $\cV_X^\dfg$ where all but finitely many of the $\cF_{\lambda}$ vanish. We also let $\bV=\bC^{\infty}=\bigcup_{n \ge 1} \bC^n$ be the standard representation of $\GL_{\infty}$; evaluating a Schur functor on $\bV$ gives an equivalence between the category of polynomial functors and the category of polynomial representations of $\GL_{\infty}$.

Since we are working over a field of characteristic $0$, we can use an alternative description of tca's: they are polynomial functors from the category of vector spaces to the category of (sheaves of) commutative algebras. We let $\bA(\cE)$ denote the tca which sends a vector space $V$ to the commutative algebra $\Sym(\cE \otimes V)$.

\subsection{K-theory of relative Grassmannians} 

For a non-negative integer $r$, $\Gr_r(\cE)$ denotes the relative Grassmannian of rank $r$ quotients of $\cE$. Let $\pi_r \colon \Gr_r(\cE) \to X$ be the structure map, and let $\cQ = \cQ_r$ be the tautological quotient bundle on $\Gr_r(\cE)$ and let $\cR = \cR_r$ denote the tautological subbundle, so that we have a short exact sequence
\[
0 \to \cR \to \pi^* \cE \to \cQ \to 0
\]
on $\Gr_r(\cE)$. Define
\begin{displaymath}
i_r \colon \rK(\Gr_r(\cE)) \to \rK(\bA(\cE)), \qquad
i_r([V]) = [\rR \pi_r{}_*(V \otimes \bA(\cQ_r))].
\end{displaymath}
Let $\Lambda$ denote the ring of symmetric functions, and let $\Lambda_d$ be the space of symmetric functions which are homogeneous of degree $d$. The group $\rK(\bA(\cE))$ is naturally a $\Lambda$-module via $s_{\lambda} \cdot [M]=[\bS_{\lambda}(\bV) \otimes M]$. The following can be found in \cite[Theorem 6.19]{symu1}:

\begin{theorem} \label{thm:groth}
The maps $i_r$ induce an isomorphism of $\Lambda$-modules
\begin{displaymath}
\bigoplus_{r=0}^{\rank \cE} \Lambda \otimes \rK(\Gr_r(\cE)) \to \rK(\bA(\cE)).
\end{displaymath}
\end{theorem}

\subsection{Integration on the torus} \label{ss:integration}

Let $T$ be the diagonal torus in $\GL(d)$.  Denote by $\alpha_1, \ldots, \alpha_d$ the standard characters $T \to \bG_m$, and identify the representation ring $K(T)$ with the ring of Laurent polynomials $\bQ[\alpha_1, \alpha_1^{-1}, \dots, \alpha_d, \alpha_d^{-1}]$.  We let $f \mapsto \int_T f d\alpha$ be the projection map $\rK(T) \to \bQ$ which takes the constant term of $f$.  We let $f \mapsto \ol{f}$ be the ring homomorphism $\rK(T) \to \rK(T)$ defined by sending $\alpha_i$ to $\alpha_i^{-1}$. Define 
\[
\Delta(\alpha) = \prod_{1 \le i<j \le d} (\alpha_i-\alpha_j), \qquad \vert \Delta \vert^2 = \Delta \ol{\Delta}.
\]
Weyl's integration formula (see \cite[\S 26.2]{FH}) can then be stated as follows:  if $f, g \in \rK(T)$ are the characters of irreducible representations of $\GL(d)$, then
\begin{equation}
\label{weyl}
\frac{1}{d!} \int_T f(\alpha) g(\ol{\alpha}) \vert \Delta(\alpha) \vert^2 d \alpha=\begin{cases}
1 & \textrm{if $f=g$} \\
0 & \textrm{if $f \ne g$}
\end{cases}.
\end{equation}

\subsection{Symmetric groups and symmetric functions}

Fix a nonnegative integer $n$. Let $\lambda$ be an integer partition of $n$ and let $c_\lambda$ be the conjugacy class of permutations with cycle type $\lambda$. Also, partitions parametrize the irreducible complex representations $\bM_\lambda$. We let $\tr(c_\mu \vert \bM_\lambda)$ denote the trace of any element of $c_\mu$ acting on $\bM_\lambda$. We refer to \cite[\S 2]{expos} for basic properties and further references.

Given an integer partition $\lambda = (\lambda_1, \dots, \lambda_r)$, let $m_i(\lambda)$ denote the number of $\lambda_j$ that are equal to $i$, and set $\lambda! = \prod_i m_i(\lambda)!$ and $z_\lambda = \lambda! \prod_i i^{m_i(\lambda)}$.

For an integer $n$, let $p_n = p_n(\alpha)$ denote the power sum symmetric polynomial $\sum_{i=1}^d \alpha_i^n$.  For a partition $\lambda=(\lambda_1, \ldots, \lambda_r)$, let $p_{\lambda}$ denote the polynomial $p_{\lambda_1} \cdots p_{\lambda_r}$. We allow $d$ to be infinite. The notation $s_\lambda$ is reserved for the Schur function. By \cite[7.17.5, 7.18.5]{stanley}, we have
\begin{align} \label{eqn:s-to-p}
s_\lambda = \sum_{|\mu|=|\lambda|} \tr(c_\mu \vert \bM_\lambda) z_\mu^{-1} p_\mu.
\end{align}

\section{Formal characters} \label{s:fchar}

{\it In this section, we will assume all schemes $X$ have an ample line bundle $\cL$.}\footnote{Recall that this means that for every coherent sheaf $\cF$, $\cF \otimes \cL^n$ is generated by global sections for $n \gg 0$. This is in particular satisfied if $X$ is a quasi-projective variety, or an affine scheme.} 

Let $V$ be an object of $\cV^{\dfg}_X$. Recall that $V$ decomposes as $\bigoplus_{\lambda} V_{\lambda} \otimes \bS_{\lambda}(\bV)$ where $V_{\lambda} \in \Mod_X^{\fgen}$. We define the {\bf formal character} of $V$ as
\begin{displaymath}
\Theta_V = \sum_{\lambda} [V_{\lambda}] s_{\lambda},
\end{displaymath}
where, as usual, $[V_{\lambda}]$ is the class of $V_{\lambda}$ in $\rK(X)$ and $s_{\lambda} \in \Lambda$ is the Schur function. Thus $\Theta_V$ is a potentially infinite series whose terms belong to $\Lambda \otimes \rK(X)$. Since $\Lambda$ is the polynomial ring in the complete homogeneous symmetric functions $\{s_n\}_{n \ge 1}$, one can also think of $\Theta_V$ as a power series in these variables with coefficients in $\rK(X)$. When $X$ is a point, $\dim(V_{\lambda})$ is the class of the vector space $V_\lambda$ in $\rK(X) \cong \bZ$.

We write $\rK'(X)$ for the Grothendieck group of vector bundles on $X$. The group $\rK(X)$ is a module over $\rK'(X)$, and there is a natural map $\rK'(X) \to \rK(X)$.

\begin{remark}[Splitting principle]
Given a symmetric polynomial $f(x_1, \dots, x_n)$ and a vector bundle $\cE$ of rank $n$ on $X$, then we define $f([\cE]) \in \rK'(X)$ as follows: if $\cE$ has a filtration whose associated graded is $\cL_1 \oplus \cdots \oplus \cL_n$, where the $\cL_i$ are line bundles, then $f([\cE]) = f([\cL_1], \dots, [\cL_n])$. In general, consider the relative flag variety $\pi \colon {\bf Flag}(\cE) \to X$. In this case, $\pi^* \cE$ has a filtration by line bundles, and we define $f([\cE]) := \rR\pi_*(f([\pi^*\cE]))$. (See \cite[\S 3.2]{fulton}; note that when $\cE$ is a line bundle, ${\bf Flag}(\cE) = X$ and $\pi^*$ and $\pi_*$ are both identity maps.)
\end{remark}

We now introduce some notation needed to state our main result on formal characters. For $k \ge 0$, let $\sigma_k = \sum_{n \ge k} \binom{n}{k} s_n$. (In the definition of $\sigma_0$ we use the convention $s_0=1$.) For a partition $\lambda$, we put
\begin{displaymath}
\sigma^{\lambda} = \prod_{n \ge 1} \sigma_n^{m_n(\lambda)}
\end{displaymath}
where $m_n(\lambda)$ is the multiplicity of $n$ in $\lambda$. 

\begin{lemma}
The elements $\sigma_0, \sigma_1, \dots$ are algebraically independent over $\bQ$.
\end{lemma}

\begin{proof}
Let $m_{d_0,\ldots,d_r}$ be the monomial $(\sigma_0-1)^{d_0} \sigma_1^{d_1} \cdots \sigma_r^{d_r}$. Let $f_{d_0,\ldots,d_r}(k)$ denote the sum of the coefficients of all $s_{\lambda}$ with $\vert \lambda \vert =k$. Then $f_{d_0,\ldots,d_r}(k)$ is a polynomial of degree $d_1+2d_2+\cdots+rd_r$. Moreover, the degree of the leading term (under the usual grading where $\deg s_\lambda = |\lambda|$) of $m_{d_0,\ldots,d_r}$ is $d_0+d_1+2d_2+ \cdots +rd_r$. Hence any linear dependence among the $m$'s can be made homogeneous if we assign $m_{d_0,\ldots,d_r}$ the bidegree $(d_0, d_1 + 2d_2 + \cdots + rd_r)$.

Since $\sigma_0-1$ is a nonzerodivisor, it suffices to handle the case $d_0=0$. But now we can simply observe that the $\sigma_k$'s with $k \ge 1$ are related to the $s_n$'s with $n \ge 1$ by an upper triangular change of variables, and the $s_n$'s are algebraically independent. 
\end{proof}

Let $\wt{\Lambda} = \Lambda \otimes \bZ[\sigma_0, \sigma_1, \dots]$ be the polynomial ring in the $s_n$'s and $\sigma_n$'s. Let $m^{(r)}_{\lambda}$ be the monomial symmetric polynomial in $r$ variables associated to the partition $\lambda$, and define 
\[
M^{(r)}_{\lambda}(x_1,\dots,x_r)=m^{(r)}_{\lambda}(x_1-1, \dots, x_r-1).
\]
For a proper map $\pi \colon Y \to X$, we define a bilinear pairing
\begin{displaymath}
\langle, \rangle \colon \rK'(Y) \times \rK(Y) \to \rK(X), \qquad \langle [\cE], [\cF] \rangle = [\rR \pi_*(\cE \otimes \cF)].
\end{displaymath}
We define a map
\begin{displaymath}
\begin{split}
\theta_r \colon \Lambda \otimes \rK(\Gr_r(\cE)) &\to \wt{\Lambda} \otimes \rK(X), \\[5pt]
s_{\mu} \otimes [\cF] &\mapsto s_\mu \sum_{\substack{\ell(\lambda) \le r}} \sigma^{\lambda} \cdot \sigma_0^{r-\ell(\lambda)} \cdot \langle M^{(r)}_{\lambda}([\cQ]), [\cF] \rangle.
\end{split}
\end{displaymath}
Here $\cQ$ is the tautological quotient bundle on $\Gr_r(\cE)$. We will see below that this sum is finite: the maximum value of $\lambda_1$ for which the term can be non-zero is bounded by a function of $\dim(X)$ and $\rank(\cE)$. We can now state our main result:

\begin{theorem} \label{thm:fchar}
Let $M \in \rD^b_\fgen(\bA(\cE))$, and write $[M]=\sum_{r=0}^{\rank \cE} c_r$ with $c_r \in \Lambda \otimes \rK(\Gr_r(\cE))$ per Theorem~\ref{thm:groth}. Then
\begin{displaymath}
\Theta_M = \sum_{r=0}^{\rank \cE} \theta_r(c_r).
\end{displaymath}
In particular, $\Theta_M$ is a polynomial in the $s_n$'s and $\sigma_n$'s with coefficients in $\rK(X)$.
\end{theorem}

Before proving the theorem, we note a few corollaries. First, in Lemma~\ref{lem:theta-inj} below, we show that $\theta_r$ is injective. Since everything in the image of $\theta_r$ has total degree $r$ in the $\sigma$'s, it follows that the images of the $\theta_r$'s are linearly independent. We conclude:

\begin{corollary}
The map
\begin{displaymath}
\Theta \colon \rK(\bA(\cE)) \to \wt{\Lambda} \otimes \rK(X)
\end{displaymath}
is injective. Thus if $M$ and $N$ are two finitely generated $\bA(\cE)$-modules then $[M]=[N]$ in $\rK(\bA(\cE))$ if and only if $[M_{\lambda}]=[N_{\lambda}]$ in $\rK(X)$ for all $\lambda$.
\end{corollary}

Let $F^{\bullet} \rK(\Gr_r(\cE))$ be the filtration by codimension of support. We also consider the topological filtration $F_{\rm top}^\bullet \rK'(\Gr_r(\cE))$ on $\rK'(\Gr_r(\cE))$ which is a ring filtration (see \cite[\S V.3]{fulton-lang}). In particular, $M_\lambda^{(r)}([\cQ]) \in F^{|\lambda|}_{\rm top}\rK'(\Gr_r(\cE))$; by definition, given a subvariety of dimension $< |\lambda|$, $M_\lambda^{(r)}([\cQ])$ can be represented by a complex whose homology is $0$ along that subvariety. It follows that $\langle M_\lambda^{(r)}([\cQ]), [\cF] \rangle=0$ if $\cF$ has support dimension $<\vert \lambda \vert$. 

For the following, let $\Mod_{A,\le r}$ be the subcategory of $A$-modules which are locally annihilated by powers of the $r$th determinantal ideal $\fa_r$, and let $T_{>r} \colon \Mod_A \to \Mod_A / \Mod_{A,\le r}$ be the localization functor. It admits a section functor (right adjoint) $S_{>r} \colon \Mod_A/\Mod_{A,\le r} \to \Mod_A$ and we set $\Sigma_{>r} = S_{>r} \circ T_{>r}$. This is a left exact functor, and induces a functor $\rR \Gamma_{>r}$ on $\rD(A)$. Next, we define $\Gamma_{\le r}$ to be the functor that assigns a module $M$ to the maximal submodule which is locally annihilated by $\fa_r$. This is also left exact and induces a functor $\rR \Gamma_{\le r}$ on $\rD(A)$. Finally, we let $\rD^b_\fgen(A)_r$ be the full subcategory of $\rD^b_\fgen(A)$ of objects which are sent to $0$ by $\rR\Gamma_{\le r-1}$ and $\rR \Sigma_{> r}$. See \cite[\S 6.1]{symu1} for more details.

\begin{corollary} \label{cor:dim-bound}
Suppose $M \in \rD^b_{\fgen}(\bA(\cE))_r$ and $[M]$ belongs to $\Lambda \otimes F^k \rK(\Gr_r(\cE))$. Then $\Theta_M$ has the form $\sum_{\vert \lambda \vert \le k} c_{\lambda} \sigma_0^{r-\ell(\lambda)} \sigma^{\lambda}$, where $c_{\lambda} \in \Lambda \otimes \rK(X)$.
\end{corollary}

The point of the corollary is that, if one gives $\sigma_n$ degree $n$, then the degree of $\Theta_M$ is related to the support dimension of $M$ on $\Gr_r(\cE)$. Of course, if $M$ is not in $\rD^b_{\fgen}(\bA(\cE))_r$ then one can apply the corollary separately to each projection $\rR \Sigma_{\ge r} \rR \Gamma_{\le r}(M)$.

We now begin with the proof of the theorem. We start with a few lemmas. If $V \in \cV_X^{\dfg}$ is $\cO_X$-flat, we let $\Theta'_V$ be defined like $\Theta_V$ except we use the class of $V_{\lambda}$ in $\rK'(X)$. Under the natural map $\rK'(X) \to \rK(X)$, we have that $\Theta'_V$ maps to $\Theta_V$. Thus $\Theta'_V$ contains more information than $\Theta_V$, when it is defined. We note that if $V, W \in \cV_X^{\dfg}$ and $V$ is $\cO_X$-flat then $\Theta_{V \otimes W}=\Theta'_V \cdot \Theta_W$, where the multiplication on the right side uses the $\rK'(X)$-module structure on $\rK(X)$.

\begin{lemma} \label{lem:fchar1}
Let $\cL$ be a line bundle on $Y$. Let $\ell=[\cL]-1 \in \rK'(Y)$ and $d=\dim(Y)$. Then
\begin{displaymath}
\Theta'_{\bA(\cL)} = \sum_{m=0}^d \ell^m \sigma_m.
\end{displaymath}
\end{lemma}

\begin{proof}
By \cite[Corollary V.3.10]{fulton-lang}, $\ell^{d+1}=0$. We have 
\begin{displaymath}
\Theta'_{\bA(\cL)}
= \sum_{n \ge 0} [\cL]^n s_n
= \sum_{n \ge 0} (\ell+1)^n s_n
= \sum_{n \ge 0} \sum_{m=0}^d \binom{n}{m} \ell^m s_n
= \sum_{m=0}^d \ell^m \left[ \sum_{n \ge m} \binom{n}{m} s_n \right]. \qedhere
\end{displaymath}
\end{proof}

\begin{lemma} \label{lem:fchar2}
Let $\cQ$ be a locally free coherent sheaf of rank~$r$ on $Y$. Then
\begin{displaymath}
\Theta'_{\bA(\cQ)} = \sum_{\substack{\ell(\lambda) \le r}} M^{(r)}_{\lambda}([\cQ]) \sigma^{\lambda} \sigma_0^{r-\ell(\lambda)}.
\end{displaymath}
\end{lemma}

\begin{proof}
By the splitting principle, we may assume $[\cQ]=[\cL_1]+\cdots+[\cL_r]$ for line bundles $\cL_i$. Put $\ell_i=[\cL_i]-1$. Since $[\bA(\cQ)] = \prod_{i=1}^r [\bA(\cL_i)]$, we use Lemma~\ref{lem:fchar1} to get
\begin{displaymath}
\Theta'_{\bA(\cQ)}
= \prod_{i=1}^r \Theta'_{\bA(\cL_i)} 
= \prod_{i=1}^r \left[ \sum_{n=0}^d \ell_i^n \sigma_n \right]
= \sum_{\substack{\lambda_1 \le d,\\ \ell(\lambda) \le r}} m^{(r)}_{\lambda}(\ell_1, \ldots, \ell_r) \sigma^{\lambda} \sigma_0^{r-\ell(\lambda)}. 
\end{displaymath}
Since $\ell_i^{d+1}=0$, the bound $\lambda_1 \le d$ in the index for the sum is superfluous. We note that $m_{\lambda}^{(r)}(\ell_1, \ldots, \ell_r)=M_{\lambda}^{(r)}([\cQ])$, by definition.
\end{proof}

\begin{proof}[Proof of Theorem~\ref{thm:fchar}]
Fix $r$, let $Y=\Gr_r(\cE)$, let $\pi \colon Y \to X$ be the structure map, and let $\cQ$ be the tautological bundle on $Y$. By \cite[Corollary~6.16]{symu1}, it suffices to prove the result for $M=\rR \pi_*(V \otimes \bA(\cQ))$ with $V \in \cV_Y^{\fgen}$. In fact, since everything in sight is $\cV$- or $\Lambda$- linear, it suffices to treat the case $V \in \Mod_Y^{\fgen}$. By Lemma~\ref{lem:fchar2}, we have 
\begin{displaymath}
\Theta_{V \otimes \bA(\cQ)} = \Theta_V \cdot \Theta'_{\bA(\cQ)} = \sum_{\substack{\ell(\lambda) \le r}} M^{(r)}_{\lambda}([\cQ]) [V] \sigma^{\lambda} \sigma_0^{r-\ell(\lambda)}.
\end{displaymath}
Since formation of formal characters is compatible with derived pushforward, we find
\begin{displaymath}
\Theta_M = \rR \pi_*(\Theta_{V \otimes \bA(\cQ)}) = \sum_{\substack{\ell(\lambda) \le r}} \langle M^{(r)}_{\lambda}([\cQ]), [V] \rangle \sigma^{\lambda} \sigma_0^{r-\ell(\lambda)}.
\end{displaymath}
Now, in the decomposition $[M]=\sum_{r=0}^{\rank \cE} c_r$ with $c_r \in \Lambda \otimes \rK(\Gr_r(\cE))$, we have $c_i=0$ for $i \ne r$ and $c_r=[V]$. Thus the above is exactly equal to $\sum_{r=0}^{\rank \cE} \theta_r(c_r)$, which proves the theorem.
\end{proof}

\begin{lemma} \label{lem:theta-inj}
The map $\theta_r$ is injective.
\end{lemma}

\begin{proof}
By definition of $\wt{\Lambda}$, if $\sum_i s_{\mu^i} \otimes [\cF_i]$ maps to $0$, then $\langle M^{(r)}_\lambda([\cQ]), [\cF] \rangle = 0$ for all $\lambda$ with $\ell(\lambda) \le r$. Via a change of basis, this implies that $\langle [\bS_\lambda(\cQ)], [\cF] \rangle = 0$ for all such $\lambda$.

Also, $\rK(\Gr_r(\cE))$ is generated by $[\bS_{\lambda^\dagger}(\cR^*) \otimes \pi^* \cG]$ where $[\cG] \in \rK(X)$, $\lambda \subseteq r \times (d-r)$, and we have $\langle [\bS_{\lambda^\dagger}(\cR^*) \otimes \pi^* \cG], [\bS_\mu(\cQ)] \rangle = \delta_{\lambda, \mu} [\cG]$ whenever $\lambda, \mu \subseteq r \times (d-r)$ by \cite[Corollary~A.3]{symu1}. In particular, write $[\cF] = \sum_\lambda [\bS_{\lambda^\dagger}(\cR^*) \otimes \pi^* \cF_\lambda]$. Applying $\langle -, [\bS_\mu(\cQ)] \rangle$, we conclude that $[\cF_\mu] = 0$ for all $\mu$, so $[\cF] = 0$.
\end{proof}

In \cite[\S 7]{symu1}, we define an equivalence of categories
\begin{displaymath}
\sF_{\cE} \colon \rD^b_{\fgen}(\bA(\cE))^{\op} \to \rD^b_{\fgen}(\bA(\cE^{\vee}))
\end{displaymath}
called the ``Fourier transform.'' To close this section, we examine how $\sF_{\cE}$ affects formal characters. In what follows, we write $\sF$ in place of $\sF_{\cE}$.

The definition of $\sF$ depends on the choice of a dualizing complex $\omega_X$ on $X$, so we fix a choice now. We let $\bD$ be the induced duality on $\rD^b_{\fgen}(X)$ and $\rK(X)$ and $\rK'(X)$. Let $(-)^\ddag \colon \Lambda \to \Lambda$ be the ring homomorphism defined by $s_{\lambda}^\ddag=(-1)^{\vert \lambda \vert} s_{\lambda^{\dag}}$. We extend $(-)^\ddag$ to infinite formal linear combinations of $s_{\lambda}$'s in the obvious manner.

\begin{lemma}
We have the identity
\begin{displaymath}
\sum_{n \ge 0} \sigma_n^\ddag t^n = \left[ \sum_{n \ge 0} \sigma_n t^n \right]^{-1}.
\end{displaymath}
In particular, $\sigma_0^\ddag=\sigma_0^{-1}$ and $\sigma_n^\ddag \in \wt{\Lambda}[1/\sigma_0]$ for all $n \ge 0$.
\end{lemma}

\begin{proof}
We have
\begin{displaymath}
\sum_{n \ge 0} \sigma_n t^n = \sum_{n,k \ge 0} \binom{k}{n} t^n s_k = \sum_{k \ge 0} (1+t)^n s_n,
\end{displaymath}
and so
\begin{displaymath}
\sum_{n \ge 0} \sigma_n^\ddag t^n = \sum_{k \ge 0} (1+t)^n (-1)^n s_{1^n}.
\end{displaymath}
Now, we have the basic identity
\begin{displaymath}
\left( \sum_{n \ge 0} s_n \right) \left( \sum_{n \ge 0} (-1)^n s_{1^n} \right) = 1.
\end{displaymath}
Since $\Lambda$ is graded with $s_n$ and $s_{1^n}$ of degree $n$, the identity continues to hold if we replace $s_n$ by $(1+t)^n s_n$ and $s_{1^n}$ by $(1+t)^n s_{1^n}$.
\end{proof}

Our main result on the Fourier transform is:

\begin{proposition} \label{prop:fchar-fourier}
Let $M \in \rD^b_{\fgen}(\bA(\cE))$. Then $\Theta_{\sF(M)} = \bD(\Theta'_{\bA(\cE)} \cdot \Theta_M^{\ddag})$.
\end{proposition}

\begin{proof}
In \cite[\S 7.1.1]{symu1}, we define the Koszul duality functor $\sK=\sK_{\cE}$. Let $T(M)=M \otimes^{\rL}_{\bA(\cE)} \cO_X$, and let $T(M)_i$ denote the degree $i$ piece of this object, regarded as a complex in $\cV_X$. By \cite[Prop~7.1]{symu1}, we have
\begin{displaymath}
\rH^n(\sK(M))=\bigoplus_{i \in \bZ} \Tor^{\bA(\cE)}_{i-n}(M, \cO_X),
\end{displaymath}
and so
\begin{displaymath}
\Theta_{\sK(M)}=\sum_{i,n \in \bZ} (-1)^{i+n} \Theta_{T(M)_i}.
\end{displaymath}
We define $\sF(M)$ by applying $\bD$ and $(-)^{\dag}$ to $\sK(M)$. Note that the above sum is $\Theta_{T(M)}$ with the terms of odd degree multiplied by $-1$; when we apply $(-)^{\dag}$ to this, we get exactly $\Theta_{T(M)}^{\ddag}$. We thus have the identity
\begin{displaymath}
\Theta_{\sF(M)}=\bD(\Theta^{\ddag}_{T(M)}).
\end{displaymath}
Now, the complex $\lw^{\bullet}(\cE \otimes \bV) \otimes M$ computes $T(M)$. The formal character of this complex is exactly $(\Theta'_{\bA(\cE)})^{\ddag} \cdot \Theta_M$. We thus have the identity
\begin{displaymath}
\Theta_{T(M)} = (\Theta'_{\bA(\cE)})^{\ddag} \cdot \Theta_M.
\end{displaymath}
Combining with the previous equation, we obtain the stated result.
\end{proof}

\begin{corollary}
Suppose $\cE$ is trivial of rank $n$. Then $\Theta_{\sF(M)}=\sigma_0^n \cdot \bD(\Theta^{\ddag}_M)$.
\end{corollary}

\begin{proof}
In this case, $\Theta'_{\bA(\cE)}=(\Theta'_{\bA(\bC)})^n$, and it is clear that $\Theta'_{\bA}(\bC)=\sigma_0 \otimes 1 \in \wt{\Lambda} \otimes \rK'(X)$.
\end{proof}

\section{The main theorems on Hilbert and Poincar\'e series}

\subsection{Standard Hilbert series} \label{s:hilbstd}

Pick $V \in \cV_X^{\dfg}$, which we think of as a sequence of $S_n$-representations on coherent $\cO_X$-modules, and decompose it as $V = \bigoplus_\lambda V_\lambda \otimes \bM_\lambda$. We define its {\bf Hilbert series} by
\begin{displaymath}
\rH_V(t) = \sum_{\lambda} [V_{\lambda}] \dim(\bM_{\lambda}) \frac{t^{\vert \lambda \vert}}{\vert \lambda \vert!},
\end{displaymath}
where $[V_{\lambda}]$ is the class of $V_{\lambda}$ in $\rK(X)$. Thus $\rH_V(t)$ is a power series with coefficients in the group $\rK(X)$. In \cite{delta}, the following theorem was proved:

\begin{theorem}
Suppose $X$ is a point and that $M$ is a finitely generated $\bA(E)$-module, where $E$ is a $d$-dimensional vector space. Then $\rH_M(t) = \sum_{i=0}^d p_i(t) e^{it}$ for some polynomials $p_i(t)$.
\end{theorem}

In this section, we generalize this theorem, and examine how the form of $\rH_M(t)$ relates to the structure of $M$. To a large extent, we answer \cite[Question~3]{delta}.

For a vector bundle $\cE$ on $X$, define
\begin{align*}
\mu_{r,0} \colon \rK(\Gr_r(\cE)) &\to \rK(X)\\
[\cF] &\mapsto \sum_{\substack{\lambda_1 \le d,\\ \ell(\lambda) \le r}} \langle m_\lambda([\cQ]^*), [\cF] \rangle \frac{t^{|\lambda|}}{\lambda_1! \cdots \lambda_r!} e^{rt}.
\end{align*}
We define a map
\begin{displaymath}
\mu_r \colon \Lambda \otimes \rK(\Gr_r(\cE)) \to \rK(X)[t], \qquad
\mu_r(s_{\lambda} \otimes x) = t^{\vert \lambda \vert} \mu_{r,0}(x).
\end{displaymath}
Let $\wh{\Lambda} = \prod_{d \ge 0} \Lambda_d$. From \cite[\S 7.8]{stanley}, we have a ring homomorphism
\begin{align*}
\ex \colon \wh{\Lambda} &\to \bQ[\![t]\!]\\
f &\mapsto \sum_{n \ge 0} (\text{coefficient of $x_1x_2 \cdots x_n$ in $f$}) \frac{t^n}{n!}.
\end{align*}
In particular, we have
\begin{align*}
\ex(s_\lambda) = \dim(\bM_\lambda) \frac{t^{|\lambda|}}{|\lambda|!},\qquad 
\ex(S_k) = \frac{t^k}{k!} e^t,
\end{align*}
and hence $\ex(\Theta_M) = \rH_M(t)$. From the above discussion and the definitions of $\theta_r$ and $\mu_r$, we have $\ex \circ \theta_r = \mu_r$.

Our main theorem on Hilbert series is then:

\begin{theorem} \label{thm:stdhilbert}
Let $M$ be a finitely generated $\bA(\cE)$-module, and write $[M]=\sum_{r=0}^d c_r$ with $c_r \in \Lambda \otimes \rK(\Gr_r(\cE))$ per Theorem~\ref{thm:groth}. Then
\begin{displaymath}
\rH_M(t) = \sum_{r=0}^d \mu_r(c_r) e^{rt}.
\end{displaymath}
\end{theorem}

\begin{proof}
By Theorem~\ref{thm:fchar}, we have $\Theta_M = \sum_{r=0}^d \theta_r(c_r)$. Now apply $\ex$ to both sides.
\end{proof}

One consequence of this theorem is that the coefficient of $e^{rt}$ in $\rH_M(t)$ depends only on the projection $\rR \Pi_r(M)$ of $M$ to $\rD^b_{\fgen}(A)_r$. We now examine how the Fourier transform interacts with Hilbert series.

\begin{proposition} \label{prop:hilbduality}
Let $M$ be a finitely generated $\bA(\cE)$-module. Then
\begin{displaymath}
\rH_{\sF_\cE(M)}(t) = \bD(e^{dt} \rH_M(-t)).
\end{displaymath}
Thus if $\rH_M(t)=\sum_{r=0}^d p_r(t) e^{rt}$, with $p_r(t) \in \rK(X)[t]$, then $\rH_{\sF_\cE(M)}(t)=\sum_{r=0}^d p^*_{d-r}(-t) e^{rt}$, where $p^*_r(t)$ is obtained from $p_r(t)$ by applying $\bD$ to its coefficients.
\end{proposition}

\begin{proof}
Apply $\ex$ to Proposition~\ref{prop:fchar-fourier}.
\end{proof}

\subsection{Enhanced Hilbert series} \label{ss:enhanced}

Let $\lambda$ be a partition. Recall that $m_i(\lambda)$ is the number of times $i$ occurs in $\lambda$. Define
\begin{displaymath}
t^{\lambda} = t_1^{m_1(\lambda)} t_2^{m_2(\lambda)} \cdots, \qquad \lambda! = m_1(\lambda)! m_2(\lambda)! \cdots
\end{displaymath}
and
\begin{displaymath}
X_{\lambda}(t) = \sum_{\vert \mu \vert = \vert \lambda \vert} \tr(c_{\mu} \vert \bM_{\lambda}) \cdot \frac{t^{\mu}}{\mu!}.
\end{displaymath}
Thus $X_{\lambda}(t)$ is a polynomial in the variables $t_i$ that encodes the character of $\bM_{\lambda}$.

\begin{lemma} \label{lem:s-h-isom}
The map
\begin{displaymath}
\varphi \colon \Lambda \otimes \bQ \to \bQ[t_i], \qquad s_{\lambda} \mapsto X_{\lambda}
\end{displaymath}
is an isomorphism of rings.
\end{lemma}

\begin{proof}
The $p_n$ are algebraically independent generators for $\Lambda \otimes \bQ$, so as can be seen from \eqref{eqn:s-to-p}, $\phi$ can be described as $p_n \mapsto nt_n$, which is evidently a ring isomorphism.
\end{proof}

We can also extend $\varphi$ to a map
\[
\varphi \colon \wh{\Lambda} \otimes \bQ \to \bQ[\![t_i]\!].
\]

Now let $V \in \cV^{\dfg}_X$. Thinking of $V$ in the symmetric group model, we can decompose it as $V = \bigoplus_\lambda V_\lambda \otimes \bM_\lambda$. We define the {\bf enhanced Hilbert series} of $V$ by
\begin{displaymath}
\wt{\rH}_V(t) = \sum_{\lambda} [V_{\lambda}] X_{\lambda}(t).
\end{displaymath}
This is a power series in the variables $t_i$ with coefficients in $\rK(X)$. This series was introduced in \cite{symc1} as an improvement of the Hilbert series, the idea being that it records the character of the $S_n$-representation of $V_n$ rather than just its dimension. Clearly, $\wt{\rH}_V(t)$ is obtained from the formal character $\Theta_V$ by applying the ring isomorphism $\varphi$. Thus the two invariants contain the same information and are just packaged in a different manner. Our results on the formal character can be easily translated to the language of enhanced Hilbert series. We just state the rationality theorem here. For $k \ge 0$, put $T_k=\sum_{n \ge k} \binom{n}{k} t_k$. It is a simple exercise to show that $\varphi(\sigma_n)$ has the form $p(T) \exp(T_0)$ where $p(T)$ is a polynomial in the $T_k$ with $k \ge 1$. Applying this to Theorem~\ref{thm:fchar} yields:

\begin{theorem} \label{thm:enhanced-rational}
Let $M$ be a finitely generated $\bA(\cE)$-module. Then
\begin{displaymath}
\wt{\rH}_M(t) = \sum_{r=0}^d p_i(t,T) e^{rT_0}
\end{displaymath}
where $p_i(t,T)$ is a polynomial in $t_1, t_2, \ldots$ and $T_1, T_2, \ldots$ with coefficients in $\rK(X)$.
\end{theorem}

\begin{remark}
One obtains the usual Hilbert series $\rH_M(t)$ from the enhanced Hilbert series $\wt{\rH}_M(t)$ by setting $t_i=0$ for $i \ge 2$. Thus the above theorem recovers our earlier rationality result for the non-enhanced Hilbert series.
\end{remark}

\subsection{Elementary proof of Theorem~\ref{thm:enhanced-rational}} \label{ss:elementary}

We now give an elementary proof (i.e., not using the results of \cite{symu1}) of our main theorem on enhanced Hilbert series.  The proof follows the proof of rationality of the usual Hilbert series given in \cite[\S 3.1]{delta}.  Namely, we express $\wt{\rH}_M(t)$ in terms of the $T$-equivariant Hilbert series of $M(\bC^d)$, where $T$ is the standard maximal torus in $\GL(d)$ for $d$ sufficiently large.  Once this expression is obtained, a formal manipulation gives the theorem.

Let $M$ be a finitely generated $A$-module, where $A$ is a tca finitely generated in degree 1. We assume (without loss of generality) that $M$ is a finitely generated module over $A=\bA(U)$ for some finite dimensional vector space $U$. Fix an integer $d \ge \max(\ell(A), \ell(M))$.

Let $V$ be a vector space (possibly infinite dimensional) on which $T$ acts. Write $V=\bigoplus V_{\alpha}$, where the sum is over the characters $\alpha$ of $T$ and $V_{\alpha}$ denotes the $\alpha$-weight space. Assuming each $V_{\alpha}$ is finite dimensional, we define the $T$-equivariant Hilbert series of $V$ by
\begin{displaymath}
\rH_{V, T}(\alpha)=\sum_{\alpha} \dim(V_{\alpha}) \alpha.
\end{displaymath}
Each character $\alpha$ is a monomial in the $\alpha_i^{\pm 1}$, and so $H_{V, T}(\alpha)$ can be regarded as a formal series in these variables. Note that if $V$ comes from a polynomial representation of $\GL(d)$ then no $\alpha_i^{-1}$'s appear in $\rH_{V, T}(\alpha)$.

Regard $A$ and $M$ as Schur functors and evaluate on the vector space $\bC^d$. We obtain a finitely generated $\bC$-algebra $A(\bC^d)$ and a finitely generated $A(\bC^d)$-module $M(\bC^d)$, equipped with compatible actions of $\GL(d)$. We can form the $T$-equivariant Hilbert series of $M(\bC^d)$, which we denote by $\rH_{M(\bC^d), T}(\alpha)$. By the comments of the last paragraph, this is a formal series in the $\alpha_i$. A simple argument with equivariant resolutions over $A(\bC^d)$ (see \cite[Lemma~3.3]{delta}), shows that
\begin{displaymath}
\rH_{M(\bC^d), T}(\alpha)=\frac{h(\alpha)}{\prod_{i=1}^d (1-\alpha_i)^n},
\end{displaymath}
where $h(\alpha)$ is a polynomial in the $\alpha_i$'s and $n=\dim(U)$.

Define
\begin{displaymath}
K(t, \alpha)=\sum_{\lambda} p_{\lambda}(\alpha) \frac{t^{\lambda}}{\lambda!},
\end{displaymath}
where the sum is taken over all partitions $\lambda$. The following lemma gives a different expression for $K(t, \alpha)$.

\begin{lemma}
\label{enh:3}
For a partition $\lambda$, let $s_{\lambda}(\alpha) \in K(T)$ denote the character of $\bS_{\lambda}(\bC^d)$.  Then
\begin{displaymath}
K(t, \alpha)=\sum_{\lambda} s_{\lambda}(\alpha) \wt{\rH}_{\bM_{\lambda}}(t),
\end{displaymath}
where the sum is over all partitions $\lambda$ with $\ell(\lambda) \le d$.
\end{lemma}

\begin{proof}
By \cite[Corollary 7.17.4]{stanley}, we have
\begin{displaymath}
p_{\mu} = \sum_{\lambda} \tr(c_{\mu} \vert \bM_\lambda) s_{\lambda}
\end{displaymath}
as elements of $K(T)$, where the sum is over partitions $\lambda$ of $\vert \mu \vert$ with $\ell(\lambda) \le d$. Multiplying by $t^{\mu}/\mu!$ and summing over $\mu$, we find
\begin{displaymath}
K(t, \alpha)=\sum_{\lambda, \mu} \tr(c_{\mu} \vert \bM_\lambda) s_\lambda \frac{t^{\mu}}{\mu!}
=\sum_{\lambda} s_{\lambda} \wt{\rH}_{\bM_{\lambda}}. \qedhere
\end{displaymath}
\end{proof}

The following is the key lemma, relating $\wt{\rH}_M(t)$ to the more accessible object $\rH_{M(\bC^d), T}$.

\begin{lemma} \label{enh:1}
We have
\begin{displaymath}
\wt{\rH}_M(t)=\frac{1}{d!} \int_T \rH_{M(\bC^d), T}(\alpha) K(t, \ol{\alpha}) \vert \Delta(\alpha) \vert^2 d\alpha.
\end{displaymath}
\end{lemma}

\begin{proof}
Let $a_{\lambda}$ be the multiplicity of $\bM_{\lambda}$ in $M$. Note that $a_{\lambda}=0$ if $d<\ell(\lambda)$, by definition of $d$. We have
\begin{displaymath}
M=\bigoplus_{\lambda} \bM_{\lambda}^{\oplus a_{\lambda}}, \qquad \textrm{and} \qquad
\wt{\rH}_M(t)=\sum_{\lambda} a_{\lambda} \wt{\rH}_{\bM_{\lambda}}(t).
\end{displaymath}
On the other hand,
\begin{displaymath}
M(\bC^d)=\bigoplus_{\lambda} \bS_{\lambda}(\bC^d)^{\oplus a_{\lambda}}, \qquad \textrm{and} \qquad
\rH_{M(\bC^d), T}(\alpha)=\sum_{\lambda} a_{\lambda} s_{\lambda}(\alpha).
\end{displaymath}
We therefore have
\begin{displaymath}
\begin{split}
\wt{\rH}_M(t) &= \sum_{\lambda} a_{\lambda} \wt{\rH}_{\bM_{\lambda}}(t)\\
&= \frac{1}{d!} \int_T \left( \sum_{\lambda} a_{\lambda} s_{\lambda}(\alpha) \right) \left( \sum_{\mu} s_{\mu}(\ol{\alpha}) \wt{\rH}_{\bM_{\mu}}(t) \right) \vert \Delta(\alpha) \vert^2 d \alpha\\
&= \frac{1}{d!} \int_T \rH_{M(\bC^d), T}(\alpha) K(t, \ol{\alpha}) \vert \Delta(\alpha) \vert^2 d\alpha.
\end{split}
\end{displaymath}
To go from the first line to the second, we used Weyl's integration formula \eqref{weyl}. Both sums in the second equation are over partitions with $\ell \le d$; in particular, both $s_{\lambda}$ and $s_{\mu}$ are characters of (non-zero) irreducible representations of $\GL(d)$. To go from the second line to the third, we used Lemma~\ref{enh:3}.
\end{proof}

We need one more lemma before proving the theorem.

\begin{lemma}
\label{enh:2}
Let $u_1, \ldots, u_d$ be indeterminates.  We then have
\begin{displaymath}
\int_T \frac{1}{\prod_{i=1}^d (1-\alpha_i u_i)} K(t, \ol{\alpha}) d\alpha
=\prod_{i=1}^d \exp \left( \sum_{n=1}^{\infty} u_i^n t_n \right).
\end{displaymath}
\end{lemma}

\begin{proof}
For $x \in \bZ_{\ge 0}^d$, let $\alpha^x$ denote $\alpha_1^{x_1} \cdots \alpha_d^{x_d}$, and similarly define $u^x$. Then we have
\begin{displaymath}
\begin{split}
\int_T \frac{1}{\prod_{i=1}^d (1-\alpha_i u_i)} K(t, \ol{\alpha}) d\alpha
&=\sum_{x} \textrm{(the coefficient of $\alpha^x$ in $K(t, \alpha)$)} u^x \\
&=\sum_{x, \lambda} \textrm{(the coefficient of $\alpha^x$ in $p_{\lambda}$)} \frac{u^x t^{\lambda}}{\lambda!}.
\end{split}
\end{displaymath}
Here $x$ varies over $\bZ_{\ge 0}^d$, while $\lambda$ varies over all partitions.  Now, a simple computation shows that the coefficient of $\alpha^x$ in $p_{\lambda}$ is given by
\begin{displaymath}
\sum_{\mu_1, \ldots, \mu_d} \frac{\lambda!}{\mu_1! \cdots \mu_d!},
\end{displaymath}
where the sum is taken over partitions $\mu_1, \ldots, \mu_d$ such that $\mu_1 \cup \cdots \cup \mu_d=\lambda$ and $\vert \mu_i \vert=x_i$.  We thus find
\begin{align*}
\sum_{x, \lambda} \textrm{(the coefficient of $\alpha^x$ in $p_{\lambda}$)} \frac{u^x t^{\lambda}}{\lambda!}
&=\sum_{\mu_1, \ldots, \mu_d} u_1^{\vert \mu_1 \vert} \cdots u_d^{\vert \mu_d \vert} \frac{t^{\mu_1} \cdots t^{\mu_d}}{\mu_1! \cdots \mu_d!} \\
&= \left( \sum_{\mu} u_1^{\vert \mu \vert} \frac{t^{\mu}}{\mu!} \right) \cdots \left( \sum_{\mu} u_d^{\vert \mu \vert} \frac{t^{\mu}}{\mu!} \right).
\end{align*}
Now, we have
\begin{displaymath}
\begin{split}
\sum_{\mu} z^{\vert \mu \vert} \frac{t^{\mu}}{\mu!}
&=\sum_{n_1 \ge 0,\ n_2 \ge 0,\ \ldots} z^{n_1+2n_2+\cdots} \frac{t_1^{n_1}}{n_1!} \frac{t_2^{n_2}}{n_2!} \cdots \\
&= \prod_{k=1}^{\infty} \left[ \sum_{n \ge 0} z^{kn} \frac{t_k^n}{n!} \right]
= \prod_{k=1}^{\infty} \exp(z^k t_k).
\end{split}
\end{displaymath}
Combining the previous two equations gives the stated result.
\end{proof}

\begin{proof}[Proof of Theorem~\ref{thm:enhanced-rational}]
Put
\begin{displaymath}
\phi_0(x, t)=\sum_{i \ge 1} x^i t_i,
\end{displaymath}
and, more generally,
\begin{displaymath}
\phi_n(x, t)=\sum_{i \ge 1} i(i-1) \cdots (i-n+1) x^{i-n} t_i=\frac{d^n}{dx^n} \phi_0(x, t).
\end{displaymath}
We have $\phi_n(1, t)=T_n$ and $\phi_n(0, t)=n! t_n$ (with $t_0$ interpreted as 0).  Lemma~\ref{enh:2} can be restated as
\begin{align} \label{eqn:K-exp}
\int_T \frac{1}{\prod_{i=1}^d (1-\alpha_i u_i)} K(t, \ol{\alpha}) d\alpha=\exp(\phi_0(u_1, t)+\cdots+\phi_0(u_d, t)).
\end{align}

Finally, suppose we are given an expression
\begin{displaymath}
\int_T \frac{p(\alpha)}{\prod_{i=1}^d (1-\alpha_i)^{n_i}} K(t, \ol{\alpha}) d\alpha
\end{displaymath}
where $p(\alpha)$ is a polynomial. We claim that it is a polynomial in $\exp(T_0)$ and $\{T_i, t_i \mid i \ge 1\}$. This claim proves the theorem, as we know that $\wt{H}_M(t)$ can be expressed as such an integral by Lemma~\ref{enh:1}.

We now prove the claim. When $n_i=0$ for all $i$, the integral in question is a polynomial in the $t_i$, so there is nothing to show in that case. By linearity and applying polynomial long division in each variable separately, we may assume that $p(\alpha)$ is a sum of monomials $\alpha_1^{m_1} \cdots \alpha_d^{m_d}$ where $m_i < n_i$ for all $i$. 
Now apply $(\frac{\partial}{\partial u_1})^{n_1-1} \cdots (\frac{\partial}{\partial u_d})^{n_d-1}$ to \eqref{eqn:K-exp} and then set $u_i=1$ for all $i$. The result is a polynomial in $\exp(T_0)$ and $\{T_i \mid i \ge 1\}$, which can be seen by using the right side of \eqref{eqn:K-exp}. The left side becomes a scalar multiple of 
\[
\int_T \frac{\alpha_1^{n_1-1} \cdots \alpha_d^{n_d-1}}{\prod_{i=1}^d (1-\alpha_i)^{n_i}} K(t, \ol{\alpha})d\alpha.
\]
For any $m_i \le n_i$, we then also know that
\[
\int_T \frac{\prod_{i=1}^d \alpha_i^{m_i-1} (1-\alpha_i)^{n_i-m_i}}{\prod_{i=1}^d (1-\alpha_i)^{n_i}} K(t, \ol{\alpha})d\alpha
\]
is a polynomial in $\exp(T_0)$ and $\{T_i \mid i \ge 1\}$. Finally, for fixed $n_i>0$, note that the set 
\[
\left\{ \prod_{i=1}^d \alpha_i^{m_i-1} (1-\alpha_i)^{n_i-m_i}\ \bigg|\ 1 \le m_i \le n_i \right\}
\]
is a basis for the span of monomials in $\alpha_1^{k_1} \cdots \alpha_d^{k_d}$ such that $0 \le k_i < n_i$. Hence by linearity and combining our earlier reduction, we conclude that our original expression is a polynomial in $\exp(T_0)$ and $\{T_i, t_i \mid i \ge 1\}$.
\end{proof}

\subsection{Poincar\'e series} \label{ss:poincare}

Given an $A$-module $M$, define its {\bf Poincar\'e series} by
\begin{displaymath}
\rP_M(t,q) = \sum_{n \ge 0} (-q)^n \rH_{\Tor_n^A(M, \bC)}(t).
\end{displaymath}
Setting $q=1$ and multiplying by $\rH_A(t)=e^{dt}$ recovers the Hilbert series $\rH_M(t)$. Note that the Poincar\'e series has non-trivial information about the $A$-module structure of $M$, whereas the Hilbert series only knows about the underlying object of $\cV_X$. The Poincar\'e series does not factor through K-theory, so it is much harder to study than the Hilbert series. The following is our main result about it:

\begin{theorem}
Let $M$ be a finitely generated $A$-module. There exist $f_r(t,q) \in \rK(X)[t,q,q^{-1}]$ such that
\begin{displaymath}
\rP_M(t,q)=\sum_{r=0}^d f_r(t,q) e^{-rqt}.
\end{displaymath}
\end{theorem}

\begin{proof}
Let $N_k=\rH^{-k}(\sF(M))$, where $\sF(M)$ is the Fourier transform of $M$. A simple formal manipulation gives $\rP_M(t,q)=\sum_{k \ge 0} (-q)^{-k} \rH_{N_k}(-qt)$. The result now follows from the fact that each $N_k$ is finitely generated and only finitely many $N_k$ are non-zero \cite[Theorem~7.7]{symu1} and the structure theorem for $\rH_{N_k}$ (Theorem~\ref{thm:stdhilbert}).
\end{proof}

There is much more one could say about Poincar\'e series and variants using the methods of this paper: for example, one could define an ``enhanced Poincar\'e series'' or a Poincar\'e-variant of the formal character, and prove results about them.

\section{A categorification of rationality} \label{s:cat}

\subsection{Motivation}

Let $\sD \colon \cV \to \cV$ be the Schur derivative. If $F$ is a polynomial functor, then $\sD F$ is the functor assigning to a vector space $V$ the subspace of $F(V \oplus \bC)$ on which $\bG_m$ acts through its standard character (this copy of $\bG_m$ acts by multiplication on $\bC$ only). In terms of sequences of symmetric group representations, the Schur derivative takes a sequence $M=(M_n)_{n \ge 0}$ to the sequence $\sD M$ given by $(\sD M)_n=M_{n+1} \vert_{S_n}$. See \cite[\S 6.4]{expos} for more details. (There the Schur derivative is denoted $\bD$, but this conflicts with our notation for duality functors.)

Let $A=\bA(E)$ and suppose $M$ is a finitely generated $A$-module. One observes that $\rH_{\sD(M)}(t) = \frac{d}{dt} \rH_M(t)$, and so the Schur derivative can be thought of as a categorification of $\frac{d}{dt}$. There is a natural map $E \otimes M \to \sD(M)$. For a subspace $V$ of $E$, let $\partial_V(M)$ be the 2-term complex $[V \otimes M \to \sD(M)]$, where the differential is just the restriction of the previously mentioned map to $V \otimes M$. Then $\rH_{\partial_V(M)}(t) = (\frac{d}{dt} - d) \rH_M(t)$, where $d=\dim(V)$, and so $\partial_V$ is a categorification of the operator $\frac{d}{dt}-d$.

The main result about $\rH_M(t)$ is that it is a polynomial in $t$ and $e^t$. It follows that $\rH_M(t)$ is annihilated by an operator of the form $\prod_{i=1}^n (\frac{d}{dt}-d_i)$ for non-negative integers $d_i$. The discussion of the previous paragraph thus suggests a stronger result, namely, that there exists subspaces $V_1, \ldots, V_n$ such that $\partial_{V_1} \cdots \partial_{V_n}$ annihilates $M$. We will show that this is indeed the case. This is a categorification of the rationality theorem for $\rH_M(t)$.

One may view this as an analogue of the existence of a system of parameters for a finitely generated module. For some discussion of this analogy, see \cite[Remark 5.4.2]{symc1}.

\subsection{The main theorem}

Let $A=\bA(E)$ and let $M$ be an $A$-module. There is then a natural map $E \otimes M \to \sD(M)$. For a subspace $V$ of $E$, we let $\partial_V(M)$ be the 2-term complex $[V \otimes M \to \sD(M)]$. More generally, for a complex $M$ of $A$-modules, we let $\partial_V(M)$ be the cone on the map $V \otimes M \to \sD(M)$. This defines an endofunctor of $\rD^b_{\fgen}(A)$. (To actually get a functor, one should work with the homotopy category of injective objects.) We define a {\bf differential operator} on $\rD^b_{\fgen}(A)$ to be a finite composition of endofunctors of the form $\partial_V$. We say that $M$ {\bf satisfies a differential equation} if there exists a differential operator $\partial$ such that $\partial(M)=0$.

\begin{theorem} \label{thm:diffeq}
Every object of $\rD^b_{\fgen}(A)$ satisfies a differential equation.
\end{theorem}

\begin{lemma}
Any two differential operators commute $($up to isomorphism$)$. In particular, if $\partial$ is a differential operator and $M$ satisfies a differential equation then $\partial(M)$ also satisfies a differential equation.
\end{lemma}

\begin{lemma}
In an exact triangle in $\rD(A)$, if two terms satisfy differential equations then so does the third.
\end{lemma}

Let $r \ge 0$ be an integer, let $Y=\Gr_r(E)$, let $\pi \colon Y \to \Spec(\bC)$ be the structure map, let $\cQ$ be the tautological bundle on $Y$, and let $B=\bA(\cQ)$. For $V \subset E$, we define the differential operator $\partial_V$ on $B$-modules just as for $A$-modules, and we define a differential operator on $\rD(B)$ to be a composition of $\partial_V$'s.

\begin{lemma}
Suppose $M \in \rD^b(B)$ satisfies a differential equation and $\cF \in \cV^{\fgen}_Y$. Then $\cF \otimes^{\rL}_{\cO_Y} M$ also satisfies a differential equation.
\end{lemma}

\begin{proof}
By induction on Tor-dimension, we can assume $\cF$ is $\cO_Y$-flat. We have 
\[
\sD(\cF \otimes M)= (\sD(\cF) \otimes M) \oplus (\cF \otimes \sD(M)).
\]
The map $E \otimes (\cF \otimes M) \to \sD(\cF \otimes M)$ maps into the factor $\cF \otimes \sD(M)$. We thus see that 
\[
\partial_V(\cF \otimes M) = (\cF \otimes \partial_V(M)) \oplus (\sD(\cF) \otimes M).
\]
We thus see that if $\partial$ is a differential operator annihilating $M$ then $\partial(\cF \otimes M)$ is a direct sum of objects of the form $\cF' \otimes M'$, where $\cF'$ has strictly lower degree than $\cF$ and $M'$ is the result of applying some differential operator to $M$. Since each $M'$ satisfies a differential equation, the result follows by induction on the degree of $\cF$.
\end{proof}

\begin{lemma}
$B$ satisfies a differential equation.
\end{lemma}

\begin{proof}
We have $\partial_V(B)=[V \to \cQ] \otimes B$. Since $[V \to \cQ]$ is a complex of $\cO_Y$-modules, it pulls out of differential operators. We thus see that if $\partial=\partial_{V_1} \cdots \partial_{V_n}$ then
\begin{displaymath}
\partial(B) = [V_1 \to \cQ] \otimes \cdots \otimes [V_n \to \cQ] \otimes B.
\end{displaymath}
For appropriate choices of the $V_i$'s, the above complex is acyclic: indeed, simply choose the $V_i$'s so that for every $y \in Y$ there is some $i$ for which $V_i(y) \to \cQ(y)$ is an isomorphism.
\end{proof}

\begin{proof}[Proof of Theorem~\ref{thm:diffeq}]
It is clear that differential operators commute with $\rR \pi_*$. We thus see that if $\cF \in \cV_Y^{\fgen}$ then the object $\rR \pi_*(\cF \otimes^{\rL}_{\cO_Y} B)$ of $\rD^b_{\fgen}(A)$ satisfies a differential equation. These objects generate $\rD^b_{\fgen}(A)$ by \cite[Corollary~6.16]{symu1} (note that by \cite[Remark~6.15]{symu1}, one can take the $\cF$ to be $\cO_Y$-flat). Thus all objects of $\rD^b_{\fgen}(A)$ satisfy a differential equation.
\end{proof}

\subsection{Differential operators and duality}

Let $M$ be a finitely generated $\bA(E)$-module. By Proposition~\ref{prop:hilbduality}, we have
\[
\rH_{\sF_E(M)}(t) = e^{dt} \rH_M(-t).
\]
For any differentiable function $f(t)$, we have the identity
\[
\left(\frac{d}{d t} - d \right) e^{dt} f(-t) = -e^{dt} \frac{d}{dt} f(-t).
\]
Hence, if $\partial$ is a differential operator that annihilates $f(t)$, then $\wt{\partial}$, obtained by applying the substitution $\frac{d}{dt} \mapsto \frac{d}{dt} - d$ to $\partial$, is a differential operator that annihilates $e^{dt} f(-t)$.

We can realize this identity using differential operators as follows. Pick a minimal free resolution $\bF_\bullet$ of $M$. Note that 
\[
\sD(V \otimes A) = (\sD(V) \otimes A) \oplus (V \otimes E \otimes A).
\]
Furthermore, for any subspace $V \subset E$, we have $\sF_E \circ \partial_V[1] = \partial_{(E/V)^*} \circ \sF_E$ (see lemma below). This implies that if $\partial$ annihilates $M$, then $\wt{\partial}$ annihilates $\sF_E(M)$, where $\wt{\partial}$ is obtained from $\partial$ by substitution of $\partial_{(E/V)^*}$ into $\partial_V$.

\begin{lemma}
For $V \subset E$ we have an isomorphism of endofunctors $\sF_E \circ \partial_V[1] \cong \partial_{(E/V)^*} \circ \sF_E$ of $\rD(A)$.
\end{lemma}

\begin{proof}
(We omit many of the details in this proof to keep it short.) By definition, we have
\begin{displaymath}
\sF_E(M) = (M \otimes_A \bK)^{R,\vee,\dag},
\end{displaymath}
where $\bK=\bK(E)$ is the Koszul complex and $(-)^\vee$ denotes vector space duality on multiplicity spaces. Using the fact that $\sD$ satisfies the Leibniz rule and $\sD(\Sym(E[1]))=E[1] \otimes \Sym(E[1])$, we find
\begin{displaymath}
\sD((M \otimes_A \bK)^{R,\vee,\dagger}) = ((\sD(M) \oplus (E[1] \otimes M)) \otimes_A \bK)^{R,\vee,\dagger} = \sF_E(\sD(M) \oplus (E[1] \otimes M)).
\end{displaymath}
The above identity neglects the differentials: in fact, the complex inside the Fourier transform is actually the cone on the canonical map $E \otimes M \to \sD(M)$, namely $\partial_E(M)$. We thus find that $\partial_{(E/V)^*}(\sF_E(M))$ is the cone on the natural map
\begin{displaymath}
\sF_E((E/V) \otimes M) = (E/V)^* \otimes \sF_E(M) \to \sD(\sF_E(M)) = \sF_E(\partial_E(M)).
\end{displaymath}
A computation shows that this map is just $\sF_E$ applied to the canonical map $\partial_E(M) \to E/V \otimes M$. (Recall that $\partial_E(M)$ is the cone of the map $E \otimes M \to \sD(M)$; the map here comes from the canonical map $E \otimes M \to E/V \otimes M$.) The cone of the map $\partial_E(M) \to E/V \otimes M$ is quasi-isomorphic to the cone of $V \otimes M \to \sD(M)$, i.e., $\partial_V(M)$. This completes the proof.
\end{proof}

\subsection{The category of differential operators} \label{ss:catdiffop}

We end this section with some additional thoughts on differential operators. It could be interesting to pursue these observations further in the future.

Write $E$ for the endofunctor $M \mapsto E \otimes M$ of $\Mod_A$, and write $\sD$ for the endofunctor given by the Schur derivative. There is then a natural transformation $E \to \sD$, from which all of the differential operator structure derives. Let $G \subset \GL(E \oplus \bV)$ be the subgroup fixing all elements of $E$, and let $\cC$ be the category of polynomial representations of $G$. Then $\cC$ is the universal tensor category having an object $X$ and a morphism $E \to X$ (one takes $X=E \oplus \bV$): that is, given any tensor category $\cD$ and a morphism $E \to X$ in $\cD$, there is a unique left-exact tensor functor $\cC \to \cD$ mapping $E \to E \oplus \bV$ to $E \to X$. (The group $G$ is an example of a ``general affine group,'' and the universal property of $\cC$, in the $d=1$ case, can be found in \cite[\S 5.4]{infrank}. Note that $E$ is simply isomorphic to a direct sum of $d$ copies of the unit object of $\cC$.)

The universal property of $\cC$ furnishes a functor $\cC \to \End(\Mod_A)$ taking $E \to E \oplus \bV$ to the transformation $E \to \sD$ discussed above. Alternatively, we can think of this functor as an action $\cC \times \Mod_A \to \Mod_A$. This action induces one on derived categories $\rD^b(\cC) \times \rD^b(A) \to \rD^b(A)$. For $V \subset E$, the 2-term complex $[V \to X]$ in $\rD^b(\cC)$ acts on $\rD^b(A)$ as the differential operator $\partial_V$. The tensor product structure on $\rD^b(\cC)$ corresponds to composition of differential operators.

Our main theorem on differential equations states that every object of $\rD^b_{\fgen}(A)$ is annihilated by some object of $\rD^b_{\fgen}(\cC)$. It would be interesting if this result could be made more precise. For example, given $M \in \rD^b_{\fgen}(A)$, what does the annihilator in $\rD^b_{\fgen}(\cC)$ look like? Is it principal, as a tensor ideal?

Another potentially interesting observation: $\cC$ is equivalent to both the category $\Mod_A^0$ of $A$-modules supported at~0 and the category $\Mod_A^{\gen}$ of ``generic'' $A$-modules (see \cite[Propositions 5.4, 5.6]{symu1}). Is this more than a coincidence? Does the action of $\cC$ on $\Mod_A$ come from, or extend to, an interesting action of $\Mod_A$ on itself? We have not fully investigated these questions.

\section{D-finiteness of Hilbert series of bounded modules} \label{s:dfin}

Recall that a power series $f \in \bC \lbb t \rbb$ is {\bf D-finite} if it satisfies a differential equation of the form
\begin{displaymath}
\sum_{i=0}^n p_i(t) \frac{d^if}{dt^i}=0,
\end{displaymath}
where each $p_i$ is a polynomial in $t$. See \cite[\S 6.4]{stanley} for some basic properties and examples. In this section, we prove the following theorem.

\begin{theorem}
\label{thm:dfin}
Let $A$ be a finitely generated tca with $A_0=\bC$ and let $M$ be a finitely generated $A$-module which is bounded. Then $\rH_M(t)$ is D-finite.
\end{theorem}

We can assume that $A$ is a polynomial tca. Recall the notation from \S\ref{ss:elementary}. 

\begin{lemma}
Suppose $\max(\ell(A), \ell(M)) < \infty$. For $d \ge \max(\ell(A), \ell(M))$, we have
\begin{displaymath}
\rH_M(t) = \int_T \rH_{M(\bC^d), T}(\alpha) e^{\sum \ol{\alpha}_i t} \vert \Delta \vert^2 d \alpha.
\end{displaymath}
\end{lemma}

\begin{proof}
$\rH_M(t)$ is obtained from the enhanced Hilbert series $\wt{\rH}_M(t)$ by doing the substitution $t_i \mapsto 0$ for $i \ge 2$ and $t_1 \mapsto t$. Apply this substitution to Lemma~\ref{enh:1}. Then $K(t, \ol{\alpha})$ becomes 
\[
\sum_\lambda \dim(\bM_\lambda) s_\lambda(\ol{\alpha}) \frac{t^{|\lambda|}}{|\lambda|!} 
= \sum_{n \ge 0} (\ol{\alpha}_1 + \cdots + \ol{\alpha}_d)^n \frac{t^n}{n!} = \exp(\ol{\alpha}_1 t + \cdots + \ol{\alpha}_d t)
\]
where the first equality comes from Schur--Weyl duality (or see \cite[Corollary 7.12.5]{stanley}).
\end{proof}

We have $A(\bC^d)=\Sym(V)$ (we warn the reader that this is evaluation of $A$ on the vector space $\bC^d$, not to be confused with $\bA(\bC^d)$) for some polynomial representation $V$ of $\GL(d)$.  Write a $T$-equivariant decomposition $V=\bigoplus_{i=1}^n \bC \alpha^{A_i}$, where $A_i \in \bZ_{\ge 0}^d$ and $\alpha^{A_i}$ means $\alpha_1^{A_{i,1}} \cdots \alpha_d^{A_{i,d}}$.  We then have
\begin{displaymath}
\rH_{M(\bC^d), T}(\alpha) \vert \Delta \vert^2=\frac{p(\alpha)}{\prod 1-\alpha^{A_i}}
\end{displaymath}
for some polynomial $p$.  It suffices to treat the case where $p$ is a monomial, say $p(\alpha)=\alpha^b$, with $b \in \bZ^d$.  We then have
\begin{displaymath}
\rH_{M(\bC^d), T}(\alpha) \vert \Delta \vert^2= \sum_{x \in \bZ_{\ge 0}^n} \alpha^{xA+b},
\end{displaymath}
and hence
\begin{displaymath}
\rH_M(t)=\sum_{x \in \bZ_{\ge 0}^n} \frac{t^{\vert xA+b \vert}}{(xA+b)!},
\end{displaymath}
where for $y \in \bZ_{\ge 0}^d$ we write $\vert y \vert=\sum_i y_i$ and $y!=\prod_i y_i!$.  Put
\begin{displaymath}
F(u_1, \ldots, u_d)=\sum_{x \in \bZ_{\ge 0}^n} u^{xA+b}=\sum_{y \in \bZ_{\ge 0}^d} \# \{ x \in \bZ_{\ge 0}^n \mid xA+b = y\} \cdot u^y.
\end{displaymath}
Then $F$ is a rational function of the $u_i$.  Indeed, let $R$ be the ring $\bC[\xi_1, \ldots, \xi_n]$ and give $R$ a $\bZ_{\ge 0}^d$ grading by $\deg(\xi_i)=A_i$. Then $F$ is the $T$-equivariant Hilbert series of the free $R$-module with one generator of degree $b$, and is therefore rational \cite[Theorem 8.20]{miller-sturmfels}.  The theorem then results from the following general result:

\begin{proposition}
Suppose that 
\[
\sum_{y \in \bZ_{\ge 0}^d} a_y u^y
\]
is a rational function of the $u_i$.  Then 
\[
\sum_{y \in \bZ_{\ge 0}^d} a_y \frac{t^{\vert y \vert}}{y!}
\]
is D-finite.
\end{proposition}

\begin{proof}
Define the Hadamard product $\ast$ on multivariate generating functions by the formula
\[
\left(\sum_{y \in \bZ_{\ge 0}^d} \alpha_y u^y \right) \ast \left(\sum_{y \in \bZ_{\ge 0}^d} \beta_y u^y \right) = \sum_{y \in \bZ_{\ge 0}^d} \alpha_y \beta_y u^y.
\]
Furthermore, we define a multivariate generating function $F$ to be D-finite if the vector space (over the field of rational functions) spanned by all partial derivatives of $F$ is finite-dimensional. By \cite[Remark 2]{lipshitz}, the Hadamard product of two D-finite generating functions is D-finite. It is clear that a rational function is D-finite and the exponential function $\exp(u_1 + \cdots + u_d)$ is also D-finite. Hence we conclude that their Hadamard product $F(u_1, \dots, u_d) = \sum_{y \in \bZ_{\ge 0}^d} a_y \frac{u^y}{y!}$ is also D-finite. Now do the substitution $u_i \mapsto t$ for $i=1,\dots, d$. By the chain rule, 
\[
\frac{d}{dt} F(t, \dots, t) = \sum_{i=1}^d \frac{\partial F}{\partial u_i},
\]
so the result is also D-finite.
\end{proof}

\section{Examples and applications} \label{s:ex}

\subsection{Hilbert series of polynomial tca's}

For an integer $n \ge 1$, define a linear map $\bQ[t_i] \to \bQ[t_i]$, denoted $f(t) \mapsto f(t^{[n]})$, by taking $t^{\lambda}$ to $t^{n\lambda}$, where $n\lambda=(n\lambda_1,n\lambda_2,\ldots)$.  We extend this map to series in the obvious manner.  The following is the main result of this section:

\begin{proposition}
\label{prop:hilbschur}
Let $V$ be a finite length object of $\cV$ concentrated in degree $d \ne 0$ and let $A=\Sym(V)$.  Then
\begin{displaymath}
\wt{\rH}_A(t) = \exp \left( \sum_{n \ge 1} n^{d-1} \wt{\rH}_V(t^{[n]}) \right).
\end{displaymath}
\end{proposition}

\begin{proof}
Both sides convert direct sums in $V$ to products, so it suffices to treat the case where $V=\bS_{\lambda}(\bC^{\infty})$. Recall from \eqref{eqn:s-to-p} that
\begin{displaymath}
s_\lambda = \sum_\mu z_\mu^{-1} \tr(c_\mu \vert \bM_\lambda) p_\mu.
\end{displaymath}
The $\GL$-character of $A$ is $\prod_T (1-x_T)^{-1}$ where the sum is over all semistandard Young tableaux $T$ of shape $\lambda$ and $x_T$ is the product of $x_i^{m_i(T)}$ where $m_i(T)$ is the multiplicity of $i$ in $T$. Apply $\log$ to this expression:
\begin{align*}
\log \left[ \prod_T (1-x_T)^{-1} \right]
&= \sum_T \log(1-x_T)^{-1}
= \sum_T \sum_{n \ge 1} \frac{x_T^n}{n}\\
&= \sum_{n \ge 1} \frac{1}{n} s_\lambda(x_1^n, x_2^n, \dots)
= \sum_{n \ge 1} \frac{1}{n} \sum_\mu z_\mu^{-1} \tr(c_\mu \vert \bM_\lambda) p_{n\mu}(x).
\end{align*}
Making the substitution $p_i \mapsto it_i$, we obtain the stated result (see Lemma~\ref{lem:s-h-isom}).
\end{proof}

\begin{example}
Take $V=(\bC^{\infty})^{\otimes d}$.  Then $\wt{\rH}_V(t)=t_1^d$, and so $\wt{\rH}_V(t^{[n]})=t_n^d$.  We thus obtain
\begin{displaymath}
\wt{\rH}_A(t)=\exp\bigg( \sum_{n \ge 1} n^{d-1} t_n^d \bigg).
\end{displaymath}
Note that for $d=1$ this is $\exp(T_0)$.
\end{example}

\begin{example}
Take $V=\Sym^2(\bC^{\infty})$, so $d=2$.  Then $\wt{\rH}_V=\tfrac{1}{2} t_1^2+t_2$, and so $\wt{\rH}_V(t^{[n]})=\tfrac{1}{2} t_n^2+t_{2n}$.  We thus obtain
\begin{displaymath}
\wt{\rH}_A(t)=\exp\bigg( \sum_{n \ge 1} \tfrac{n}{2} t_n^2 +nt_{2n}\bigg).
\end{displaymath}
In particular, $H_A(t)=\exp(\tfrac{1}{2} t^2)$.  The case $V=\lw^2(\bC^{\infty})$ is similar:  just change the $+$ sign to a $-$ sign.
\end{example}

\begin{remark}
In the notation of the proposition, the tca $A$ is generated in degree $d$, and so Theorem~\ref{thm:enhanced-rational} cannot be applied to it for $d>1$; in fact, the above calculations show that the conclusion of the theorem is false in this case.  Nonetheless, the values of $\wt{\rH}_A$ computed above are sufficiently nice that one might hope for a good generalization of Theorem~\ref{thm:enhanced-rational} which includes these cases.
\end{remark}

\subsection{Formal characters of determinantal rings}

Let $A$ be the tca $\bA(E)$, where $E$ is a $d$-dimensional vector space. Let $\fa_r \subset A$ be the $r$th determinantal ideal, generated by the representation $\lw^{r+1}(\bC^{\infty}) \otimes \lw^{r+1}(E)$. In this section, we derive a formula for the formal character of $A/\fa_r$.

Recall that we have an isomorphism of Grothendieck groups
\begin{equation} \label{eq:k}
\rK(A)=\bigoplus_{r=0}^d \Lambda \otimes \rK(\Gr_r(E)).
\end{equation}
To apply Theorem~\ref{thm:fchar}, we need to understand the class $[A/\fa_r]$ in $\rK(A)$ under this identification. The following lemma gives us what we need.

\begin{lemma}
Under the isomorphism \eqref{eq:k}, the class of $[A/\fa_r]$ corresponds to $1 \otimes [\cO_{\Gr_r(E)}]$.
\end{lemma}

\begin{proof}
Let $\cQ_r$ be the tautological bundle on $Y_r=\Gr_r(E)$, and let $\pi_r \colon Y_r \to \Spec(\bC)$ be the structure map. Let $i_r \colon \rK(\Gr_r(E)) \to \rK(A)$ be the map defined by $i_r([V])=[\rR (\pi_r)_*(V \otimes \bA(\cQ_r))]$. According to \cite[Theorem~6.19]{symu1}, the maps $i_r$ induce the isomorphism \eqref{eq:k}. Now, we have
\begin{displaymath}
i_r([\cO_{Y_r}])=[\rR (\pi_r)_*(\bA(\cQ_r))]
\end{displaymath}
and
\begin{displaymath}
\bA(\cQ_r) = \bigoplus_{\ell(\lambda) \le r} \bS_{\lambda}(\bC^{\infty}) \otimes \bS_{\lambda}(\cQ_r).
\end{displaymath}
The vector bundle $\bS_{\lambda}(\cQ_r)$ has no higher cohomology, and its space of sections is $\bS_{\lambda}(E)$. We thus find that $\rR (\pi_r)_*(\bA(\cQ_r))=A/\fa_r$, which proves the lemma.
\end{proof}

For the rest of the section, we fix $r$ and put $Y=\Gr_r(E)$. We let $\cQ$ be the tautological bundle on $Y$ and $\pi \colon Y \to \Spec(\bC)$ the structure map. From the above lemma and Theorem~\ref{thm:fchar}, we see that the formal character $\Theta_{A/\fa_r}$ is given by $\theta_r([\cO_Y])$. We now compute $\theta_r([\cO_Y])$, at least to some extent. To do this, we must compute the inner product $\langle M_\lambda^{(r)}([\cQ]), [\cO_Y] \rangle$. To start, define $S_\lambda^{(r)}(x_1,\dots,x_r) = s_\lambda^{(r)}(x_1-1, \dots, x_r-1)$ where $s_\lambda^{(r)}$ denotes the Schur polynomial in $r$ variables.

\begin{lemma}
If $\ell(\lambda) \le r$, then $\langle S_\lambda^{(r)}([\cQ]), [\cO_Y] \rangle = \dim \bS_{\lambda^\dagger}(\bC^{d-r})$.
\end{lemma}

\begin{proof}
We first verify this when $\lambda = (n)$. In that case, it follows from the definitions (or the general formula \cite[p. 47, Example I.3.10]{macdonald}) that
\[
s_n^{(r)}(x_1-1,\dots,x_r-1) = \sum_{i=0}^n (-1)^{n-i} \binom{n+r-1}{i+r-1} s^{(r)}_i(x_1,\dots,x_r),
\]
and so
\begin{displaymath}
S^{(r)}_n([\cQ])=\sum_{i=0}^n (-1)^{n-i} \binom{n+r-1}{i+r-1} s^{(r)}_i([\cQ]).
\end{displaymath}
Now, $s^{(r)}_i([\cQ])$ is simply $[\Sym^i(\cQ)]$, and so we see
\begin{displaymath}
\langle s^{(r)}_i([\cQ]), [\cO_Y] \rangle = [\rR \pi_*(\Sym^i(\cQ))] = \binom{d+i-1}{i}.
\end{displaymath}
Here we used the fact that $\pi_*(\Sym^i(\cQ))$ has dimension $\binom{d+i-1}{i}$, and the higher pushforwards vanish. We thus find
\begin{align*}
\langle S^{(r)}_n([\cQ]), [\cO_Y] \rangle &= \sum_{i=0}^n (-1)^{n-i} \binom{n+r-1}{i+r-1} \binom{d+i-1}{i}\\
&= \sum_{i=0}^n (-1)^{n} \binom{n+r-1}{n-i} \binom{-d}{i}\\
&= (-1)^n \binom{n-d+r-1}{n}\\
&= \binom{d-r}{n} = \dim \bS_{(n)^\dagger}(\bC^{d-r}).
\end{align*}
For the second and fourth equalities, we used the identities $\binom{a}{b} = \binom{a}{a-b}$ and $\binom{a+b-1}{b} = (-1)^b \binom{-a}{b}$, and the third identity is the Chu--Vandermonde identity. This proves the lemma for $\lambda=(n)$.

The above reasoning applies equally well to $r$-fold products of $S^{(r)}_n$. That is, we have
\begin{displaymath}
\langle S^{(r)}_{n_1}([\cQ]) \cdots S^{(r)}_{n_r}([\cQ]), [\cO_Y] \rangle
= \dim \big( \bS_{(n_1)^{\dag}}(\bC^{d-r}) \otimes \cdots \otimes \bS_{(n_r)^{\dag}}(\bC^{d-r}) \big)
\end{displaymath}
for any $n_1, \ldots, n_r$. The key point here is that
\begin{displaymath}
[\rR \pi_*(\Sym^{n_1}(\cQ) \otimes \cdots \otimes \Sym^{n_r}(\cQ))]
= \dim \big( \Sym^{n_1}(\bC^d) \otimes \cdots \otimes \Sym^{n_r}(\bC^d) \big).
\end{displaymath}
To see this, we first decompose the tensor product $\Sym^{n_1}(\cQ) \otimes \cdots \otimes \Sym^{n_r}(\cQ)$ into a direct sum of Schur functors $\bS_\lambda(\cQ)$ with $\ell(\lambda) \le r = \rank(\cQ)$ where the bound comes from the Pieri rule. Now Borel--Weil--Bott (see \cite[Theorem A.1]{symu1} for a convenient formulation) says that $\rR \pi_* \bS_\lambda(\cQ) = \bS_\lambda(\bC^d)$, and we get a direct sum which is the same as the tensor product $\Sym^{n_1}(\bC^d) \otimes \cdots \otimes \Sym^{n_r}(\bC^d)$ (it is important that there are at most $r$ factors in the tensor product so that all Schur functors appearing in it have at most $r$ rows, and therefore do not annihilate $\cQ$).

We now prove the lemma for general $\lambda$ by induction on dominance order (i.e., $\lambda$ dominates $\mu$ if $|\lambda|=|\mu|$ and $\lambda_1 + \cdots + \lambda_i \ge \mu_1 + \cdots + \mu_i$ for all $i$). The base case is when $\lambda$ has a single part, which we have already done. In general, $s_\lambda = s_{\lambda_1} \cdots s_{\lambda_r} + C$ where $C$ is a sum of $s_\mu$ with $\mu$ less dominant than $\lambda$. By induction, $\langle S_\mu^{(r)}([\cQ]), [\cO_Y] \rangle = \dim \bS_{\mu^\dagger}(\bC^{d-r})$ for all such $\mu$. By the previous paragraph,
\[
\langle S_{\lambda_1}([\cQ]) \cdots S_{\lambda_r}([\cQ]), [\cO_Y] \rangle = \dim ( \bS_{(\lambda_1)^{\dagger}}(\bC^{d-r}) \otimes \cdots \otimes \bS_{(\lambda_r)^{\dagger}}(\bC^{d-r}) ).
\]
Hence the statement also holds for $\lambda$.
\end{proof}

Let $K_{\lambda,\mu}$ be the Kostka number, i.e., the number of semistandard Young tableaux of shape $\lambda$ and content $\mu$, so that
\[
s_\lambda = \sum_\mu K_{\lambda, \mu} m_\mu.
\]
The matrix $(K_{\lambda,\mu})$ is invertible; let $K^{-1}_{\lambda, \mu}$ denote the entries of the inverse matrix, so that
\[
m_\lambda = \sum_\mu K^{-1}_{\lambda, \mu} s_\mu.
\]
The above identities also hold after specializing to any finite number of variables, as this operation is linear. We thus find
\[
\langle M_\lambda^{(r)}([\cQ]), [\cO_Y] \rangle = \sum_{\ell(\mu) \le r} K^{-1}_{\lambda, \mu} \langle S_\mu^{(r)}([\cQ]), [\cO_Y] \rangle = \sum_{\ell(\mu) \le r} K^{-1}_{\lambda, \mu} s_{\mu^\dagger}^{(d-r)}(1,\dots,1).
\]

We finally reach our main result:

\begin{theorem}
We have
\begin{displaymath}
\Theta_{A/\fa_r}=\sum_{\lambda \subseteq r \times d} c_{\lambda} \sigma^\lambda \sigma_0^{r-\ell(\lambda)} 
\end{displaymath}
where $c_{\lambda}$ is the integer given by
\begin{displaymath}
c_{\lambda}=\sum_{\mu \subseteq r \times (d-r)} K^{-1}_{\lambda, \mu} \dim \bS_{\mu^{\dag}}(\bC^{d-r}).
\end{displaymath}
\end{theorem}

In the case $r=1$, we can give a simpler expression for the formal character by computing directly:

\begin{proposition} \label{prop:fchar1}
We have $\Theta_{A/\fa_1}=\sum_{n=0}^{d-1} \binom{d-1}{n} \sigma_n$.
\end{proposition}

\begin{proof}
We have
\begin{displaymath}
A/\fa_1 = \bigoplus_{n \ge 0} \Sym^n(E) \otimes \Sym^n(\bC^{\infty}),
\end{displaymath}
and so
\begin{displaymath}
\Theta_{A/\fa_1} = \sum_{n \ge 0} \binom{d+n-1}{d-1} s_n,
\end{displaymath}
where the binomial coefficient is the dimension of $\Sym^n(E)$. Appealing to the identity
\begin{displaymath}
\binom{d+n-1}{d-1} = \sum_{i=0}^{d-1} \binom{d-1}{i} \binom{n}{d-i-1},
\end{displaymath}
we find
\begin{displaymath}
\Theta_{A/\fa_1} = \sum_{i=0}^{d-1} \binom{d-1}{i} \left[ \sum_{n \ge 0} \binom{n}{d-i-1} s_n \right]
=\sum_{i=0}^{d-1} \binom{d-1}{i} \sigma_{d-1-i},
\end{displaymath}
from which the stated formula follows.
\end{proof}

\subsection{Enhanced Hilbert series of determinantal rings}

Keep the same notation as the previous section. The formula for $\Theta_{A/\fa_r}$ gives a formula for $\wt{\rH}_{A/\fa_r}$ by a change of variables. In this section, we compute $\wt{\rH}_{A/\fa_r}$ directly, without using any of the theory behind Theorem~\ref{thm:fchar}.

First consider the case $r=1$. Then $A/\fa_1=\bigoplus_{n \ge 0} \Sym^n(\bC^\infty)^{\oplus \dim \Sym^n(E)}$, and so the enhanced Hilbert series is
\begin{align} \label{eqn:rank1det}
\sum_\lambda \binom{|\lambda|+d-1}{|\lambda|} \frac{t^\lambda}{\lambda!}.
\end{align}

\begin{lemma} 
\[
\sum_\lambda \binom{|\lambda|+d-1}{|\lambda|} \frac{t^\lambda}{\lambda!} = \left( \frac{\partial^{d-1}}{\partial s^{d-1}} \frac{s^{d-1}}{(d-1)!} \exp(st_1 + s^2 t_2 + s^3 t_3 + \cdots) \right) \bigg|_{s=1}.
\]
\end{lemma}

\begin{proof}
The coefficient of $t^{n_1} \cdots t^{n_r}$ in $\exp(st_1 + s^2 t_2 + \cdots)$ is 
\[
\frac{s^{n_1 + 2n_2 + \cdots + rn_r} }{n_1! \cdots n_r!}.
\]
Set $n=n_1 + 2n_2 + \cdots + rn_r$. If we multiply by $\frac{s^{d-1}}{(d-1)!}$, take the $(d-1)$st derivative, and evaluate at $s=1$, we get
\[
\binom{n+d-1}{n} \frac{1}{n_1! \cdots n_r!}.
\]
Finally, given a partition $\lambda$, and setting $n_i = m_i(\lambda)$, we have $n=|\lambda|$.
\end{proof}

Using the product rule for derivatives, this can be written as
\[
\left. \sum_{i+j=d-1} \binom{d-1}{i} \frac{s^j}{j!} \frac{d^j}{ds^j} \exp(st_1 + s^2t_2 + \cdots) \right|_{s=1}.
\]
Note that $T_i = \frac{1}{i!} \frac{d^i}{ds^i} (st_1 + s^2 t_2 + \cdots)|_{s=1}$, so we see that this expression is a polynomial in $T_1, \dots, T_{m-1}$ times $\exp(T_0)$. 

We list the first few derivatives of $\exp(f(s))$ (they are of the form $\exp(f(s)) \cdot F$ so we just list $F$):
\begin{align*}
f'(s)\\
f''(s) + f'(s)^2\\
f^{(3)}(s) + f'(s)^3 + 3f'(s)f''(s)\\
f^{(4)}(s) + 3f''(s)^2 + f'(s)^4 + 4f^{(3)}(s)f'(s) + 6f'(s)^2f''(s)
\end{align*}
In particular, we get the following expressions for $\frac{d^j}{ds^j} \exp(s t_1 + s^2 t_2 + \cdots)|_{s=1}$ (they are of the form $\exp(T_0) \cdot F$ so we just list $F$):
\begin{align*}
T_1\\
2T_2 + T_1^2\\
6T_3 + T_1^3 + 6T_1T_2\\
24T_4 + 12T_2^2 + T_1^4 + 24T_3T_1 + 12T_1^2T_2
\end{align*}
So we have the following evaluations of \eqref{eqn:rank1det} for small $d$:
{\small
\begin{align*}
\begin{array}{l|l}
d & \eqref{eqn:rank1det} \\[2pt] \hline
1 & \exp(T_0)  \rule{0pt}{12pt} \\[2pt]
2 & (T_1 + 1) \exp(T_0)\\[2pt]
3 & (\frac{1}{2} (2T_2 + T_1^2) + 2T_1 + 1) \exp(T_0)\\[2pt]
4 & (\frac{1}{6} (6T_3 + T_1^3 + 6T_1T_2) + \frac{3}{2} (2T_2 + T_1^2) + 3T_1 + 1) \exp(T_0)\\[2pt]
5 & (\frac{1}{24}(24T_4 + 12T_2^2 + T_1^4 + 24T_3T_1 + 12T_1^2T_2) +  \frac{2}{3} (6T_3 + T_1^3 + 6T_1T_2) + 3 (2T_2 + T_1^2) + 4T_1 + 1) \exp(T_0)
\end{array}
\end{align*}
}
One can verify that these quantities correspond to those in Proposition~\ref{prop:fchar1} under the appropriate change of variables. 

Using \cite[Theorem 16]{gessel}, we can give a determinantal formula for the enhanced Hilbert series of $A/\fa_r$, in general, in terms of certain modifications of \eqref{eqn:rank1det}. Define, for any $i \in \bZ$,
\[
a_i = \sum_{\lambda,\ |\lambda| \ge -i} \binom{|\lambda|+i+d-1}{|\lambda|+i} \frac{t^\lambda}{\lambda!}.
\]
Then $a_0$ is the same as \eqref{eqn:rank1det}, and the enhanced Hilbert series of $A/\fa_r$ is the determinant
\[
\det(a_{j-i})_{i,j=1}^r.
\]
From this formula, we deduce that the enhanced Hilbert series is a polynomial in $T_1, \dots, T_{d-1}$ times $\exp(rT_0)$. This gives an independent verification of the rationality theorem in this case.

\subsection{Hilbert series of invariant rings} \label{ss:invariants}

Let $E$ be a finite dimensional representation of a reductive group $G$ and let $A$ be the tca $\bA(E)^G$. Using the symmetric group model for $\cV$, $A$ is given by $A_n=(E^{\otimes n})^G$. Since $\bA(E)$ is bounded, so is $A$, and so Theorem~\ref{thm:dfin} assures us that $\rH_A$ is D-finite.

For the purposes of this section, we let $\rH_A^*(t)=\sum_{n \ge 0} \dim(A_n) t^n$. This is the non-exponential form of $\rH_A(t)$. The D-finiteness of $\rH_A$ is equivalent to the D-finiteness of $\rH_A^*$: each of the series $\sum_n t^n/n!$ and $\sum_n n! t^n$ is D-finite, and D-finiteness is preserved under Hadamard product by \cite[Remark 2]{lipshitz} (see also \cite[Theorem 6.4.12]{stanley} for a simpler reason for univariate series).

\begin{example}
Take $G=\SL(2)$ and $E=\bC^2$. Then the dimension of $(E^{\otimes n})^G$ is 0 for $n$ odd and the Catalan number $C_{n/2}$ for $n$ even. (Recall that $C_k=\frac{1}{k+1} \binom{2k}{k}$.)  We thus obtain
\begin{displaymath}
\rH^*_A(t)=\sum_{k \ge 0} C_k t^{2k}=\frac{2}{1+\sqrt{1-4t^2}}.
\end{displaymath}
We note that any algebraic function is D-finite. We have
\begin{displaymath}
\rH_A(t)=\sum_{k \ge 0} \frac{(2k)!}{k! (k+1)!} t^{2k}.
\end{displaymath}
This is (essentially) a modified Bessel function of the first kind, and satisfies the differential equation
\begin{displaymath}
t^2 \frac{d^2y}{dt^2}+3t \frac{dy}{dt}-4t^2y=0. \qedhere
\end{displaymath}
\end{example}

\begin{example} \label{ex:notalg}
Take $G=\SL(2) \times \SL(2)$ and $E=\bC^2 \otimes \bC^2$. Then $(E^{\otimes n})^G$ is the tensor square of the space $((\bC^2)^{\otimes n})^{\SL(2)}$ from the previous example, and so we find
\begin{displaymath}
\rH_A^*(t)=\sum_{k \ge 0} C_k^2 t^{2k}=\sum_{k \ge 0} \frac{1}{(k+1)^2} \binom{2k}{k}^2 t^{2k}.
\end{displaymath}
This series is not algebraic: indeed, if it were then the Hadamard product of it and the rational series $\sum_{k \ge 0} (k+1)^2 t^{2k}$ would also be algebraic \cite[Proposition 6.1.11]{stanley}, but this is the series $\sum_{k \ge 0} \binom{2k}{k}^2 t^{2k}$, which is known to be transcendental \cite[Theorem 3.3]{woodcocksharif}. We thus see that $\rH_A^*(t)$ need not always be algebraic. Thus Theorem~\ref{thm:dfin} is optimal, in a sense.
\end{example}

\begin{remark}
We point out that the D-finiteness of $\rH^*_A(t)$ can be seen directly in this case. Let $T$ be a maximal torus of $G$, and let $\alpha_1, \ldots, \alpha_r$ be characters of $T$ giving an isomorphism with $\bG_m^r$. Let $\chi$ be the character of $E$, regarded as a function on $T$ (and thus a Laurent polynomial in the $\alpha_i$'s). Let $F(t,\alpha)=\sum_{n \ge 0} t^n \chi^n = (1-t \chi)^{-1}$. This is a rational function of $t$ and the $\alpha$'s. By the Weyl integration formula (see \S \ref{ss:integration} for the $\GL(n)$ case and \cite[\S 26.2]{FH} for the general case), we have
\begin{displaymath}
\rH_A^*(t) = \frac{1}{|W|}\int_T F(t,\alpha) \vert \Delta \vert^2 d\alpha,
\end{displaymath}
where $W$ is the Weyl group of $G$ and $\Delta$ is a certain Laurent polynomial in the $\alpha_i$. The D-finiteness of $\rH_A^*(t)$ now follows from the following general fact: taking the constant term (with respect to the $\alpha_i$) of the the product of a D-finite function with a Laurent polynomial is again D-finite.
\end{remark}

\begin{remark}
There are some similar situations where one does not have D-finiteness. For $k \ge n$, the canonical map
\begin{displaymath}
((\bC^k)^{\otimes n})_{S_k} \to ((\bC^{k+1})^{\otimes n})_{S_{k+1}}
\end{displaymath}
is an isomorphism. Let $A_n$ be the stable value of this vector space (precisely, we could define it as the direct limit). The dimension of $A_n$ is the Bell number $B_n$. These numbers do not have a D-finite generating series: one has $\sum_{n \ge 0} B_n \tfrac{t^n}{n!}=e^{e^t-1}$.

This example can be made to look more like the previous ones by using the category $\Rep(\fS)$ of algebraic representations of the infinite symmetric group $\fS$, as studied in \cite{infrank}. Let $E=\bC^{\infty}$ be the permutation representation of $\fS$. Then the sequence $(A_n)_{n \ge 0}$ is identified with $\bA(E)_{\fS}$. Thus the conclusion of Theorem~\ref{thm:dfin} does not apply to $\bA(E)_{\fS}$.

We thank a referee for suggesting this example.
\end{remark}

\subsection{Generalizing character polynomials} \label{ss:gencp}

Let $E$ be a vector space of dimension $m$, and let $M$ be a finitely generated module over $\bA(E)$. Theorem~\ref{thm:enhanced-rational} tells us that the enhanced Hilbert series can be written in the form
\[
\wt{\rH}_M(t) = \sum_{r=0}^d p_i(t,T) e^{rT_0}
\]
where $p_i(t,T)$ is a polynomial in $t_1, t_2, \ldots$ and $T_1, T_2, \ldots$. In fact, using Corollary~\ref{cor:dim-bound}, we can be more precise: if we set $\deg(T_d) = d$ and $\deg(t_d)=0$, then each polynomial $p_i$ has degree $\le i(m-i)$.

We define the $k$th umbral substitution to be the linear map
\[
\bQ[t_i] \to \bQ[a_i], \qquad \prod_i t_i^{d_i} \mapsto \prod_i k^{-d_i} (a_i)_{d_i},
\]
where $(x)_d = x(x-1) \cdots (x-d+1)$. This extends to a linear map on polynomial expressions in the $t_i$ and $T_i$, and we denote this operation by $p \mapsto \downarrow_k p$. When $k=1$, this was denoted $\downarrow$ in \cite[\S 5.2]{symc1}.

\begin{proposition}
Given $M$, there exist polynomials $p_1,\dots,p_d$ in $t_1,t_2,\dots,T_1,T_2,\dots$ such that $\deg(p_i) \le i(m-i)$, and
\[
\tr(c_\lambda \vert M) = \sum_{i=1}^d i^{\ell(\lambda)} \downarrow_i p_i|_{t_k = m_k(\lambda)}
\]
when $|\lambda| \gg 0$.
\end{proposition}

\begin{proof}
Consider an expression $t_1^{\alpha_1} \cdots t_r^{\alpha_r} T_1^{\beta_1} \cdots T_s^{\beta_s} \exp(i T_0)$. The coefficient of $t^\lambda / \lambda!$ is 
\[
\sum_{v_1, \dots, v_s \in \bZ_{\ge 0}^{r}} i^{\ell(\lambda) - |\alpha| - |\beta|} \frac{\lambda!}{(\lambda-\alpha-\beta)!} \prod_{i=1}^s \prod_{j=1}^r \binom{v_{i,j}}{i}
\]
where the sum is over all $(v_1,\dots,v_s)$ such that $|v_i|=\beta_i$, and we define $|\alpha|=\alpha_1 + \cdots + \alpha_r$, $|\beta| = \beta_1 + \cdots + \beta_s$, and $(\lambda-\alpha-\beta)! = \prod_{i=1}^r (m_i(\lambda) - \alpha_i - v_{1,i} - \cdots - v_{s,i})!$ when the arguments are nonzero, and otherwise, the term does not appear. This is the same as $\downarrow_i (t_1^{\alpha_1} \cdots t_r^{\alpha_r} T_1^{\beta_1} \cdots T_s^{\beta_s})$, and hence the result follows from the form of $\wt{\rH}_M(t)$ above if we only consider $|\lambda|$ larger than $\deg(p_0(t,T))$ where we define $\deg(t_i) = i$.
\end{proof}

When $d=1$, the $T_1,T_2,\dots$ do not show up, and $\downarrow_1 p_1$ is a polynomial. This is the character polynomial discussed in \cite[\S 5.2]{symc1} and \cite[Theorem~1.5]{fimodule}. In the general case, we need to use a formal series.


\begin{thebibliography}{[CEF]}

\bibitem[CEF]{fimodule}
Thomas Church, Jordan~S. Ellenberg, Benson Farb, FI-modules: a new approach to stability for $S_n$-representations, {\it Duke Math. J.} {\bf 164}, no.~9 (2015), 1833--1910, \arxiv{1204.4533v4}.

\bibitem[Fu]{fulton} William Fulton, {\it Intersection Theory}, second edition, Springer-Verlag, Berlin, 1998. 

\bibitem[FH]{FH}
William Fulton, Joe Harris, \emph{Representation Theory: A First Course}, Graduate Texts in Mathematics, {\bf 129}, Springer--Verlag, New York, 1991.

\bibitem[FL]{fulton-lang} William Fulton, Serge Lang, {\it Riemann-Roch Algebra}, Grundlehren der Mathematischen Wissenschaften [Fundamental Principles of Mathematical Sciences] {\bf 277}, Springer-Verlag, New York, 1985.

\bibitem[Ges]{gessel}
Ira~M. Gessel, 
Symmetric functions and P-recursiveness,
{\it J. Combin. Theory Ser. A} {\bf 53} (1990), no.~2, 257--285. 

\bibitem[Lip]{lipshitz} L.~Lipshitz, 
The diagonal of a $D$-finite power series is $D$-finite,
{\it J. Algebra} {\bf 113} (1988), no.~2, 373--378. 

\bibitem[Mac]{macdonald} I. G. Macdonald, {\it Symmetric Functions and Hall Polynomials}, second edition, Oxford Mathematical Monographs, Oxford, 1995.

\bibitem[MS]{miller-sturmfels}
Ezra Miller, Bernd Sturmfels, 
{\it Combinatorial Commutative Algebra},
Graduate Texts in Mathematics {\bf 227}, Springer-Verlag, New York, 2005.

\bibitem[SS1]{symc1}
Steven~V Sam, Andrew Snowden, GL-equivariant modules over polynomial rings in infinitely many variables, 
{\it Trans. Amer. Math. Soc.} {\bf 368} (2016), 1097--1158, \arxiv{1206.2233v3}.

\bibitem[SS2]{expos}
Steven~V Sam, Andrew Snowden, Introduction to twisted commutative algebras, \arxiv{1209.5122v1}.

\bibitem[SS3]{infrank}
Steven~V Sam, Andrew Snowden, Stability patterns in representation theory, 
{\it Forum Math. Sigma} {\bf 3} (2015), e11, 108 pp., \arxiv{1302.5859v2}.

\bibitem[SS4]{catgb}
Steven~V Sam, Andrew Snowden, Gr\"obner methods for representations of combinatorial categories, {\it J. Amer. Math. Soc.} {\bf 30} (2017), 159--203, \arxiv{1409.1670v3}.

\bibitem[SS5]{symu1}
Steven~V Sam, Andrew Snowden, GL-equivariant modules over polynomial rings in infinitely many variables II, \arxiv{1703.04516v1}.

\bibitem[Sno]{delta}
Andrew Snowden, Syzygies of Segre embeddings and $\Delta$-modules, {\it Duke Math.\ J.} {\bf 162} (2013), no.~2, 225--277, \arxiv{1006.5248v4}.

\bibitem[Sta]{stanley}
Richard~P. Stanley, {\it Enumerative Combinatorics}, Vol. 2, with a foreword by Gian-Carlo Rota and appendix 1 by Sergey Fomin, Cambridge Studies in Advanced Mathematics {\bf 62}, Cambridge University Press, Cambridge, 1999.

\bibitem[WS]{woodcocksharif}
Christopher F.~Woodcock, Habib Sharif, On the transcendence of certain series, {\it J. Algebra} {\bf 121} (1989), no.~2, 364--369.

\end{thebibliography}
\end{document}